\documentclass[11pt,a4paper]{article}
\usepackage[body={154mm,229mm},top=34mm]{geometry}
\usepackage{amscd} \usepackage{mathrsfs} \usepackage{epsfig} \usepackage{amssymb} \usepackage{amsmath,bm} \usepackage{amsthm}
\usepackage{epsfig} \usepackage{verbatim} \usepackage{url} \usepackage{color} \usepackage{hyperref} \include{pspicture}

\newcommand{\ninsepsc}[3]{
\begin{figure}[h]

\begin{center}
 \scalebox{#3}{\includegraphics{#1}}
\end{center}

\vspace{-0.8cm}
\caption{\hspace{0.25cm}#2\label{f:#1}}
\end{figure}
}

\newtheorem{lemma}{Lemma}

\newtheorem{remark}{Remark}

\newtheorem{proposition}{Proposition}[section]

\usepackage{makeidx}
\usepackage{mathabx}
\usepackage{mathrsfs} 
\begin{document}
\title{Approximation of Excessive Backlog Probabilities of Two Tandem Queues}
\author{Ali Devin Sezer\footnote{Middle East Technical University, Institute of Applied Mathematics, Ankara, Turkey}}
\maketitle
\begin{abstract}
Let $X$ be the constrained random walk on ${\mathbb Z}_+^2$ 
with increments $(1,0)$, $(-1,1)$ and $(0,-1)$; $X$ represents the lengths of two queues
in tandem where arrivals are Poisson to the first queue with rate $\lambda$,
and the service times are exponentially distributed with rates $\mu_1$ and $\mu_2$;
we assume $\lambda < \mu_1, \mu_2$, i.e., $X$ is assumed stable
and $\mu_1 \neq \mu_2$ (the case $\mu_1 = \mu_2$ can be handled
by allowing $\mu_1$ converge to $\mu_2$). 
 Let $\tau_n$ be the first
time $X$ hits the line $\partial A_n = \{x:x(1)+x(2) = n \}$, i.e.,  when the sum of the components of $X$ equals $n$ for the first time.
For $x \in {\mathbb Z}_+^2, x(1) + x(2) < n$, the probability $p_n(x)= P_x( \tau_n < \tau_0)$ 
is one of the key performance
measures for the queueing system represented by $X$  (if the queues share a common buffer,
$p_n(x)$ is the probability that this buffer
overflows during the system's first busy cycle).
Let $Y$ be the random walk on ${\mathbb Z} \times {\mathbb Z}_+$ with increments $(-1,0)$, $(1,1)$ and $(0,-1)$ 
that is constrained to be positive only on its second component. Let $\tau$ be the first time that the components
of $Y$ equal each other. Let $\rho_i = \lambda/\mu_i$, $i=1,2$, denote the utilization rates of the nodes.
We derive the following explicit formula for $P_y(\tau < \infty)$:
\[
P_y(\tau < \infty) = W(y)= \rho_2^{y(1)-y(2)} + \frac{\mu_2 - \lambda}{\mu_2 - \mu_1} 
\rho_1^{ y(1)-y(2)} \rho_1^{y(2)}
+ \frac{\mu_2-\lambda}{\mu_1 -\mu_2} \rho_2^{y(1)-y(2)} \rho_1^{y(2)},
\]
$y \in {\mathbb Z}\times{ \mathbb Z}_+$,
$y(1) > y(2)$,
and show that $W(n-x_n(1),x_n(2))$
approximates $p_n(x_n)$ with relative error {\em exponentially decaying} in $n$
for
$x_n = \lfloor nx \rfloor$, $x \in {\mathbb R}_+^2$,
$0 < x(1) + x(2) < 1$.
Our analysis consists of the following steps: 1) with an affine transformation,
move the origin of the coordinate system to the point $(n,0)$ on the exit boundary $\partial A_n$;
let $n\rightarrow \infty$ to remove the constraint on the $x(2)$ axis.
; this step gives the limit {\em unstable} /{\em transient} 
constrained random walk $Y$ that is constrained only on the $x(1)$ axis,
and reduces $P_{x}(\tau_n < \tau_0)$ to $P_y(\tau < \infty)$.
2) construct a basis of harmonic functions of $Y$ and
use this basis to apply the classical superposition principle of linear analysis
to compute $P_y(\tau < \infty).$
The construction of basis functions involve the use of conjugate points on a
characteristic surface associated with the walk $X$.
The  proof that the relative error decays exponentially in $n$
uses a sequence of subsolutions of a related Hamilton Jacobi Bellman 
equation on a manifold; the manifold consists
of three copies of ${\mathbb R}_+^2$, the zeroth
glued to the first along $\{x:x(1)=0\}$ and the first to the second
along $\{x:x(2) =0\}.$
We indicate how the ideas of the paper can be generalized
to more general processes and other exit boundaries.

~\\
\noindent
{\bf MSC classes:} 60G50, 60G40, 60F10, 60J45,
{\bf Keywords:} Large deviations, constrained random walks, buffer overlow, queueing systems, exit times, harmonic systems
\end{abstract}


\section{Introduction and definitions}\label{s:intro}
Let $X$ be a random walk
with independent and identically distributed increments 
$\{I_1,I_2,I_3,...\}$,
constrained to remain in ${\mathbb Z}_+^2$:
\begin{align*}
X_0 &=x \in {\mathbb Z}_+^2, ~~~X_{k+1} \doteq X_k + \pi ( X_k,I_k ), k=1,2,3,...\\
\pi(x,v) &\doteq \begin{cases} v, &\text{ if } x +v  \in {\mathbb Z}^2_+ \\
				  0      , &\text{otherwise,}
\end{cases}, 
\end{align*}
\[
I_k \in \{(1,0),(-1,1),(0,-1) \}, 
P(I_k = (1,0)) = \lambda,
P(I_k = (-1,1)) = \mu_1,
P(I_k = (0,-1)) = \mu_2.
\]
Let $\partial_i \doteq \{ x \in {\mathbb Z}^2: x(i) = 0 \}$, $i=1,2$,
denote the constraining boundaries of the process and let $\sigma_i \doteq \inf \{k: X_k \in \partial_i\}$, $i=1,2$,
denote the first time $X$ hits these boundaries.
The components of $X$ represents
the number of customers at jump times of a Jackson network consisting of two 
tandem queues.

We assume that $X$ is stable, i.e., $\lambda  <\mu_1,\mu_2$.
We also assume $\mu_1 \neq \mu_2$;  subsection \ref{ss:equal}
comments on the case $\mu_1 = \mu_2$.
Define
\begin{equation}\label{e:defA}
A_n = \left\{ x \in {\mathbb Z}_+^2: x(1) + x(2) \le n \right\}
\end{equation}
\index{$A_n$}
and its boundary
\begin{equation}\label{e:defBoundaryA}
\partial A_n =
  \left\{ x \in {\mathbb Z}_+^2: x(1) + x(2) = n \right\}.
\end{equation}
Let $\tau_n$ be the first time $X$ hits $\partial A_n$:
\begin{equation}\label{d:taun}
\tau_n \doteq \inf\{k: X_k \in \partial A_n\}.
\end{equation}
Define
$p_n \doteq P_x( \tau_n < \tau_0)$, i.e., 
the probability that, starting from an initial state $x \in A_n$,
the total number of
customers in the system reaches $n$ before the system empties. 
The set $A_n$
models a systemwide shared buffer of size $n$.
If we measure
time in  the number of independent cycles that restart each time $X$ hits $0$, $p_n$ is 
the probability that the current cycle finishes successfully (i.e., without a buffer
overflow). 
One can change the domain $A_n$ to model other buffer structures, e.g.,
$ \{ x \in {\mathbb Z}_+^2: x(i) \le n, i=1,2\}$
models separate buffers of size $n$ for each queue in the system. 
The present work focuses on 
the domain $A_n$. The basic ideas of the paper
apply to other domains, and we comment on this in Section \ref{s:conclusion}.
Let $Y$ be the random walk on ${\mathbb Z} \times {\mathbb Z}_+$ with increments $(-1,0)$, $(1,1)$ and $(0,-1)$ 
that is constrained to be positive only on its second component. Let $\tau$ be the first time that the components
of $Y$ equal each other (the relation between $X$ and $Y$ is explained in the paragraphs below).
In Section \ref{s:twodim} we
derive the following explicit formula for $P_y(\tau < \infty)$:
\begin{equation}\label{d:Wstar}
P_y(\tau < \infty) = W^*(y)
\doteq 
\left( \rho_2^{y(1)-y(2)} 
- \frac{\mu_2-\lambda}{\mu_2 -\mu_1} \rho_2^{y(1)-y(2)} \rho_1^{y(2)}\right)
+ \frac{\mu_2 - \lambda}{\mu_2 - \mu_1} 
\rho_1^{y(1)-y(2)} \rho_1^{y(2)}, ~~
\end{equation}
$y \in {\mathbb Z}^2_+$, $y(1) > y(2).$
Fix $x \in \{x \in {\mathbb R}_+^2: 0 < x(1) + x(2) < 1\}$ and define
$x_n = \lfloor nx \rfloor$.
In Section \ref{s:convergence2} we show that $W^*(n-x_n(1),x_n(2))$ approximates
$p_n(x_n)$, 
with relative error {\em exponentially vanishing } in $n$ (see Proposition \ref{t:guzel}).
The following paragraphs note prior literature and results relating
to the approximation of $p_n$
and summarize the analysis that lead to the results summarized above.

For a stable $X$, the event $\{ \tau_n <\tau_0\}$ rarely happens
and, conditioned on a fixed initial point $x$, its probability $p_n$ decays 
exponentially with buffer size $n$. Because $X$ is Markov,  $p_n$, as a
function of the initial point $x$, satisfies a system of linear equations, see
\eqref{e:linear}. As $n$ gets large, the number of unknowns grow like
$n^2$ and 
it becomes infeasable to solve the system exactly. 
\cite{GlassKou, ignatiouk2000large} compute
the large deviation limit of $p_n(x)$, for $x=(1,0)$, as
\[
\lim_{n\rightarrow \infty} -\frac{1}{n}\log p_n( (1,0)) = \min(-\log \rho_1,-\log \rho_2),
\]
where $\rho_i = \lambda/\mu_i.$
Because $p_n$ is a small probability, i.e., the probability of a rare event, a natural
idea is to use importance sampling to approximate it via simulation.
To the best of our knowledge, the article \cite{parekh1989quick}
is the first to study the optimal IS simulation of the two tandem walk model for the boundary
$\partial A_n$; it was observed in \cite{parekh1989quick} that static changes of measure
implied by optimal large deviation sample paths may not lead to optimal IS changes of measure because
of the constraining boundaries of the process. \cite{parekh1989quick} introduced boundary layers to the
problem and allowed the change of measure to depend on whether the 
process is in these layers.
It was observed in \cite{GlassKou} that a simple change of measure implied by LD analysis
(exchange the arrival rate with the smaller of the service rates)
could perform poorly for the exit boundary $\partial A_n =  \{x: x(1)+ x(2) = n \}$ for a range
of parameter values.
An asymptotically optimal change of measure for this boundary was developed in \cite{DSW} using subsolution
of a limit Hamilton Jacobi Bellman (HJB) equation; similar to
the heuristic constructions in \cite{parekh1989quick},
the change of measure developed in \cite{DSW} is dynamic, 
i.e., it depends on the position
of the process $X$; 
\cite{istrees,sezer-asymptotically,yeniDW,KDW} treats higher dimensions, 
more general
dynamics and different exit boundaries using the subsolution apprach. Let ${\bm \tau}_{\bm 0}$ denote
the first return time to the origin.
The work \cite{mcdonald1999asymptotics} proposes an alternative approximation approach to probabilities of
the type $P_{\bm 0}(\tau_n < {\bm \tau}_{\bm 0})$ for a class of models under a number of assumptions; 
the approximation idea in \cite{mcdonald1999asymptotics} 
is to replace $\tau_0$ with $\tau_\triangle$, and the initial position ${\bm 0}$ with a random initial
point on $\triangle$ with distribution $\pi_\triangle$, where $\triangle$ are the constraining boundaries
corresponding to a set of ``non-super-stable'' nodes, $\tau_\triangle$ is the first nonzero time when one these nodes
become empty, and $\pi_\triangle$ is the stationary measure of the underlying
process conditioned on $\triangle$; \cite{mcdonald1999asymptotics} and its approach are further reviewed in Section \ref{s:review}.
There is a vast literature on the analysis and simulation of rare events of constrained
random walks, in particular, and on
the analysis of constrained random walks in general
 \cite{
alanyali1998large,
aldous2013probability, 
anantharam, 
asmussen2007stochastic, 
asmussen2008applied, 
atar1999large,  
blanchet2006efficient, 
blanchet2009lyapunov,
blanchet2009rare, 
blanchet2013optimal, 
BoerNicola02, 
borovkov2001large,
Changetal, 
collingwood2011networks, 
comets2007distributed, 
comets2009large, 
dai2011reflecting, 
de2006analysis, 
dean2009splitting,
dieker2005asymptotically,
duphui-is1, 
dupuis1995large,
flajolet1986evolution, 
foley2005large, 
foley2012constructing, 
frater1991optimally, 
guillotin2006dynamic, 
Iglehart74, 
ignatiouk2010martin, 
ignatyuk1994boundary,
juneja2006rare, 
JunejaNicola,
JunejaRandhawa, 
KDW, 
kobayashi2013revisiting, 
KroeseNicola,
kurkova1998martin, 
louchard1991probabilistic, 
louchard1994random, 
maier1991colliding, 
maier1993large,
mcdonald1999asymptotics, 
miretskiy2008state, 
miyazawa2009tail, 
miyazawa2011light, 
MR1110990, 
MR1335456,
MR1736592, 
ney1987markov,
nicola2007efficient, 
parekh1989quick, 
ridder2009importance, 
Rubetal04,
rubino2009rare, 
sezer2009importance, 
sezer-asymptotically, 
yao1981analysis, 
yeniDW}.
Section \ref{s:review} reviews a number of these works
in relation to the results and the techniques of the current work.

One way to think about the LD analysis is as follows. 
$p_n$ itself
decays to $0$, which is trivial. To get a nontrivial limit 
transform $p_n$ to $V_n \doteq -\frac{1}{n}\log p_n$;
using convex duality, one can write
the $-\log$ of an expectation as an optimization
problem involving the relative entropy \cite{dupell2} and thus
$V_n$ can be interpreted 
as the value function of a discrete time stochastic optimal control 
problem. 
The LD analysis
consists of the law of large numbers limit analysis of this control problem;
the limit problem is a deterministic optimal control problem whose value function
satisfies a first order Hamilton Jacobi Bellman equation (see
 \eqref{e:HJB0} of Section \ref{s:convergence2}).
Thus, LD analysis amounts to the computation of the
limit of a {\em convex  transformation} of the problem.

We will use another, an {\em affine}, transformation of $X$ for the limit analysis.
The proposed transformation is very simple:  
{\em observe $X$ from the exit boundary.}
For the two tandem walk 
the most natural vantage point
on the exit boundary $\partial A_n$ turns out to be the corner $(n,0).$
Therefore, we transform the process thus
\begin{equation}\label{e:defTn}
Y^n \doteq T_{n}(X), ~~~ T_{n}: {\mathbb R}^2\rightarrow {\mathbb R}^2, 
T_{n}(x) \doteq y, ~~~, y(2) = x(2), y(1) = n-x(1).
\end{equation}
\index{T@$T_n$ the affine transformation mapping $X$ to $Y^n$}
$T_{n}$ is affine and its inverse equals itself.
$Y^n$, i.e., the process $X$ as observed from the corner $(n,0)$,
 is a constrained process on the domain 
$\Omega_Y^n\doteq  (n - {\mathbb Z}_+) \times {\mathbb Z}_+$.
$T_{n}$ maps
the set $A_n$ to
$B_n \subset \Omega_Y^n$, $B_n \doteq T_{n} (A_n)$,
the corner $(n,0)$
to the origin of $\Omega_Y^n$; 
the exit boundary  $\partial A_n$
to 
$\partial B_n \doteq \{ y \in \Omega_Y^n, y(1)=y(2)\}$ and
finally the constraining
boundary $\{x \in {\mathbb Z}_+^2, x(1)=0\}$ to
\[
\{ y \in {\mathbb Z}_+^2: y(i) = n \}.
\]
As $n\rightarrow \infty$ 
the last boundary vanishes and $Y^n$ converges to
the 
limit process $Y$ on the domain
$\Omega_Y \doteq  {\mathbb Z} \times {\mathbb Z}_+$
and
the set $B_n$ to 
\begin{equation}\label{e:defB}
B \doteq \left\{ y \in \Omega_Y, y(1) \ge y(2)\right\}.
\end{equation}
The exit boundary for the limit problem is
\begin{equation}\label{e:defBb}
\partial B = \left\{ y \in \Omega_Y, y(1) = y(2) \right\};
\end{equation}
the limit stopping time
\begin{equation}\label{e:deftau}
\tau \doteq \inf \{k: Y_k \in \partial B\}
\end{equation}
is the first time $Y$ hits $\partial B$.
{\em 
The stability of $X$ and the vanishing of the boundary constraint on $\partial_1$ implies
that $Y$ is {\em unstable} / {\em transient}, i.e., with probability $1$ it wanders off to $\infty$}.
Therefore, in our formulation, the limit process is an {\em unstable} constrained random walk
in the same space and time scale as the original process
but with less number of constraints and the limit problem
is whether this unstable process ever hits the fixed boundary $\partial B$.

Figure \ref{f:Tn2tex} sketches these transformations.

\begin{figure}[h]
\begin{center}
\scalebox{0.7}{
\centerline{\input{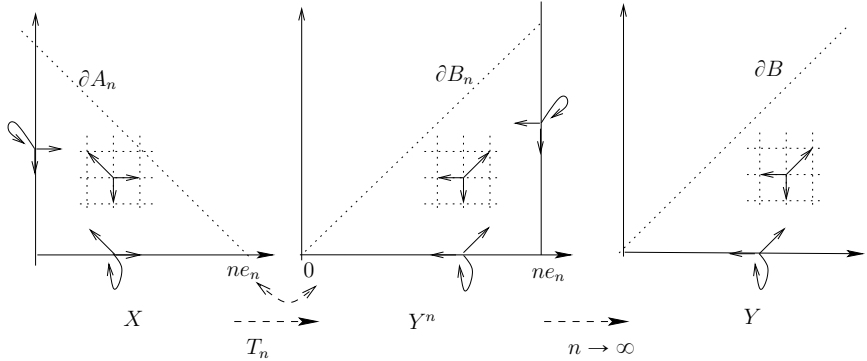}}}
\end{center}

\vspace{-0.65cm}
\caption{\hspace{0.25cm}The transformation $T_n$\label{f:Tn2tex}}
\end{figure}

Fix an initial point $y \in B$ in the new coordinates; our first
convergence result is Proposition \ref{t:c1} which says
\begin{equation}\label{e:conv1}
p_n = P_{x_n}( \tau_n < \tau_0 )
\rightarrow  P_y( \tau < \infty),
\end{equation}
where $x_n = T_n(y)$.
The proof uses the
law of large numbers and LD lowerbounds to show that the difference between the two
sides of \eqref{e:conv1} vanishes with $n$.
With \eqref{e:conv1} we see that 
the limit problem in our formulation is to compute the hitting probability of
the unstable $Y$ to 
the boundary $\partial B$.

The convergence statement \eqref{e:conv1} involves a fixed initial condition for the
process $Y$. In classical LD analysis, one specifies the initial point
in scaled coordinates as follows: 
$x_n = \lfloor n x \rfloor \in A_n $ for $x \in {\mathbb R}_+^d$.
Then the initial
condition for the $Y^n$ process will be
$y_n = T_n(x_n)$
(thus we fix not the $y$ coordinate but the
scaled $x$ coordinate). 
When $x_n$ is defined in this way, \eqref{e:conv1}
becomes a trivial statement because its both sides decay to $0$.
For this reason, Section \ref{s:convergence2} studies
the relative error
\begin{equation}\label{e:conv2}
\frac{| P_{x_n}( \tau_n < \tau_0 ) - P_{y_n}( \tau < \infty)|}{P_{x_n}(\tau_n < \tau_0)};
\end{equation}
Proposition \ref{t:guzel} says that this error converges exponentially to $0$
for the case of two dimensional tandem walk (i.e., the process $X$ shown in
Figure \ref{f:Tn2tex}).
The proof rests on
showing that the probability of the intersection of the events 
$\{\tau_n < \tau_0\}$ and $\{ \tau < \infty\}$ 
dominate the probabilities of both as $n\rightarrow \infty$. 
For this we calculate bounds in Proposition \ref{t:thm1}
on the LD decay rates of the probability
of the differences between these events using
a sequence of subsolutions of a Hamilton Jacobi Bellman equation {\em on a manifold};
the manifold consists of three copies of ${\mathbb R}_+^2$,
zeroth copy glued to the first along $\partial_1$, and the first to the second
along $\partial_2$, where $\partial_i = \{x\in {\mathbb R}_+^2: x(i) = 0 \}.$
Extension of this argument to more complex processes and
domains remains for future work.

The convergence results
\eqref{e:conv1} and \eqref{e:conv2} reduce the problem of calculation
of $P_x(\tau_n < \tau_0)$ to that of $P_y(\tau < \infty)$. 
This constitutes the first
step of our analysis and we expect it to apply more generally;
see subsection \ref{ss:possiblegen}. 

Computation of $P_y(\tau < \infty)$ 
is a static linear problem and can be attacked
with a range of ideas and methods.
Section \ref{s:twodim}
applies the {\em principle of superposition}
of classical linear analysis to the computation of $P_y(\tau < \infty)$.
The key for its application is to construct
the right class of efficiently computable basis functions to be superposed.
The construction of our basis functions goes as follows:
the distribution
of the increments of $Y$ is used to define the 
{\em characteristic polynomial}
${\bm p}:{\mathbb C}^2 \rightarrow {\mathbb C}$.
${\bm p}$ can be represented both as a rational
function and as a polynomial.
We call the $1$ level set of ${\bm p}$, the {\em characteristic surface} of $Y$
and denote it with ${\mathcal H}$, see \eqref{d:DefHcal}.
${\mathcal H}$ is, more precisely, a $1$ dimensional complex affine algebraic variety of degree $3$.
Each point on the characteristic surface ${\mathcal H}$
defines a $\log$-linear function (see Proposition \ref{p:basicharmonicZ})
that satisfies the interior harmonicity condition of $Y$ 
(i.e., defines a harmonic function of the completely unconstrained version 
of $Y$); 
similarly, each boundary of the state
space of $Y$ has an associated characteristic polynomial and surface.
${\bm p}$ can be written as a second order polynomial in each of its arguments;
this implies that most points on ${\mathcal H}$ come in conjugate pairs.
The keystone of the approach developed in Section \ref{s:twodim} is the following
observation:
{\em $\log$-linear functions defined by two points 
on ${\mathcal H}$ satisfying a given type of conjugacy relation
can be linearly combined to get nontrivial functions which satisfy
the corresponding boundary harmonicity condition
(as well as the interior one)}; see Figure \ref{f:charsurftex} and Proposition \ref{p:harmonicYtwoterms2d}.
We show that any solution to a
harmonic system gives a harmonic function for $Y$ in the form of linear combinations
of $\log$-linear functions (each vertex defines a $\log$-linear function).

There is a direct connection between the computations given
in the present paper and the Balayage operator \cite{revuz1984markov},
we point out this connection in subsection \ref{ss:partialBdet},
 Remark \ref{r:balayage}.
Section \ref{s:examples} gives a numerical example. 
The conclusion (Section \ref{s:conclusion}) discusses several directions 
for future research. Among these is the application of the approach of 
the present paper to constrained diffusion processes and the associated 
elliptic equations with
Neumann boundary conditions (subsection \ref{ss:diffusionswithdrift}).

\section{Derivation of the limit problem}\label{s:convergence1}
This section derives  the limit problem resulting from the affine transformation $T_n$.
The derivation is simple enough and therefore will be given for a more general 
setup: for the purposes of the present section
we will assume $X$ to be the embedded random walk of a $d$ 
dimensional stable Jackson network; let, as before, $I_k$ denote the unconstrained iid increments of $X$.
Define
\begin{equation}\label{e:DefIi}
{\mathcal I}_1 \in {\mathbb R}^{d \times d},~~
{\mathcal I}_1(j,k) = 0, j \neq k,~~ {\mathcal I}_1(j,j) = 1, j \neq 1,~~
{\mathcal I}_1(1,1) = -1.
\end{equation}
${\mathcal I}_1$ is the identity operator on ${\mathbb R}^d$ except that its first diagonal term
is $-1$ rather than $1$. The affine change of coordinate map will be
\begin{equation}\label{e:explicitTn}
T_n = n e_1 + {\mathcal I}_1,
\end{equation}
where $e_1 \doteq (1,0,0,...,0) \in {\mathbb R}^d.$
Define the sequence of transformed increments
\begin{equation}\label{e:defJ}
J_k \doteq {\mathcal I}_1 (I_k).
\end{equation}
The domain of the limit $Y$ process will be $\Omega_Y = {\mathbb Z} \times {\mathbb Z}_+^{d-1}$
and the limit process will have dynamics
\[
Y_{k+1} = Y_k + \pi_1(Y_k,J_k),
\]
where
\[
\pi_1(x,v) \doteq 
\begin{cases} v, &\text{ if } x +v  \in \Omega_Y,\\
	 0    , &\text{otherwise.}
\end{cases}
\]
Let
$A_n=  \{x \in {\mathbb Z}_+^d: x(1) + x(2) + \cdots + x(d) \le n \}$, $\tau_n$ is the first time $X$ hits 
$\partial A_n = \{x \in {\mathbb Z}_+^d: x(1) + x(2) + \cdots + x(d) = n \}$.
The limit exit boundary will
be $\partial B = \{y \in \Omega_Y, y(1) \ge \sum_{i=2}^d y(i) \}$,
$\tau$ is the first time $Y$ hits $\partial B.$ Set $\partial_1 = \{z \in {\mathbb Z}^d: z(1) = 0 \}$
and $\sigma_1$ will be the first time $X$ hits $\partial_1.$

Denote by ${\mathcal X}$ the law of large numbers limit of $X$ , i.e., the deterministic
function which satisfies
\begin{equation}\label{e:lln}
\lim_n P_{x_n}\left ( \sup_{k \le t_0 n} |X_k/n -  {\mathcal X}_{k/n} | > \delta \right) = 0
\end{equation}
for any $\delta > 0$ and $t_0 > 0$
where  $x_n \in {\mathbb Z}_+^d$ is a sequence of initial positions
satisfying $\frac{x_n}{n} \rightarrow \chi \in {\mathbb R}_+^d$ 
(see, e.g., \cite[Proposition 9.5]{robert2003stochastic} or 
\cite[Theorem 7.23]{chen2013fundamentals}).
The limit process starts from
${\mathcal X}_0 = \chi$,  is piecewise affine and takes values in ${\mathbb R}_+^d$;
then
 $s_t \doteq \sum_{i=1}^d {\mathcal X}_t(i)$ starts from $\sum_i \chi(i)$
is also piecewise linear and continuous (and therefore differentiable except
for a finite number of points) with 
values in ${\mathbb R}_+$. The stability and bounded iid increments of
$X$ imply that $s$ is strictly decreasing 
and 
\begin{equation}\label{e:sdecrease}
 c_1 > -\dot{s} > c_0 > 0
\end{equation}
for two constants $c_1$ and $c_0$.
These imply that ${\mathcal X}$ goes in finite time  $t_1$ to
$ 0 \in {\mathbb R}_+^d$ and remains there afterward.

Fix an initial point $y \in \Omega_Y$ for the process $Y$ and set
$x_n = T_n(y)$; it follows from the definition of $T_n$ that
\begin{equation}\label{e:llninit}
\frac{x_n}{n} \rightarrow e_1 \doteq (1,0,0,....,0) \in {\mathbb R}^d.
\end{equation}

\begin{proposition}\label{t:c1}
Let $y$ and $x_n$ be as above. Then
\[
\lim_{n\rightarrow \infty} P_{x_n}(\tau_n < \tau_0) =  P_y(\tau < \infty ).
\]
\end{proposition}
\begin{proof}
Note that $x_n \in A_n$ for $n > y(1)$.
Define
\[
M_k = \max_{l \le k } Y_l(i), ~~M^X_k = \min_{l \le k } X_l(i).
\]
$M$ is an increasing process and $M_\tau$ is the greatest that the first
component of $Y$ gets before hitting $\partial B$ (if this happens in finite time).
The monotone convergence theorem implies
\[
P_y( \tau < \infty) = \lim_{n\nearrow \infty} 
P_y( \tau < \infty, M_\tau <  n ).
\]
Thus 
\begin{equation}\label{e:decompose1}
P_y( \tau < \infty) = 
P_y( \tau < \infty, M_\tau < n )
+
P_y( \tau < \infty, M_\tau \ge n )
\end{equation}
and the second term goes to $0$ with $n$.
Decompose $P_{x_n}( \tau_n < \tau_0)$ similarly using $M^X$:
\begin{align*}
P_{x_n}( \tau_n < \tau_0 ) &= 
P_{x_n}\left( \tau_n < \tau_0, M^X_{\tau_n}  > 0  \right)
+
P_{x_n}\left( \tau_n < \tau_0, M^X_{\tau_n}  = 0 \right).\\
\intertext{ 
On the set $\{M^X_{\tau_n} >0 \}$, the process $X$ cannot reach the boundary
$\partial_1$ before $\tau_n$, therefore over this set 
1) the events $\{\tau_n < \tau_0\}$ and
$\{\tau < \infty\}$ coincide  (remember that $X$ and $Y$ are defined
on the same probability space)
2) the distribution of $(T_n(X),n-M^X)$ is the same as that of
$(Y,M)$ upto time $\tau_n.$
Therefore,}
&= P_y( \tau < \infty, M_\tau < n) + 
P_{x_n}\left( \tau_n < \tau_0, M^X_{\tau_n} =  0 \right).
\end{align*}
The first term on the right equals the first term on the right side of 
\eqref{e:decompose1}. We know that the second term in \eqref{e:decompose1}
goes to $0$ with $n$. Then  to finish our proof, it suffices to show
\begin{equation}\label{e:toproveconv1}
\lim_n 
P_{x_n}\left( \tau_n < \tau_0, M^X_{\tau_n} = 0\right)  = 0.
\end{equation}
$M^X_{\tau_n} = 0$ means that $X$ has hit $\partial_1$ before $\tau_n$.
Then the last probability equals
\begin{equation}\label{e:lastbit}
 P_{x_n}\left (  \sigma_1 < \tau_n < \tau_0\right ),
\end{equation}
which, we will now argue, goes to $0$ ($\sigma_1$ is the first time $X$
hits $\partial_1$);
\eqref{e:llninit} implies ${\mathcal X}_0 = e_1.$ 
Define $t^1 \doteq \inf\{t: {\mathcal X}_t(1)  = 0 \}$
and $t^0 \doteq \inf\{t: {\mathcal X}_t = 0 \in {\mathbb R}^d\}.$
By definition
$t^1 \le t^0 < \infty$
Now choose $t_0$ in \eqref{e:lln} to be equal to $t^0$, define
${\mathcal C}_n \doteq
\left\{ \sup_{k \le t^0 n} \in |X_k/n -  {\mathcal X}_{k/n} | > \delta \right \}$
and partition \eqref{e:lastbit} with ${\mathcal C}_n$:
\begin{align}\label{e:decomposewithCn}
P_{x_n}\left (  \sigma_1 < \tau_n < \tau_0\right) = 
P_{x_n}\left ( \{  \sigma_1 < \tau_n < \tau_0 \} \cap  {\mathcal C}_n\right)
+P_{x_n}\left( \{  \sigma_1 < \tau_n < \tau_0 \} \cap  {\mathcal C}_n^c\right).
\end{align}
The first of these goes to $0$ by \eqref{e:lln}.
The event in the second term is the following: $X$ remains at most $n\delta$ 
distance away $n {\mathcal X}$ until its $nt^0$ step, hits $\partial_1$
then  $\partial A_n$ and then $0$.
These and \eqref{e:sdecrease} imply that,
for $n$ large enough,
any sample path lying in this event can hit
$\partial A_n$ only after time $nt^0$. Thus, the  second probability on the
right side of \eqref{e:decomposewithCn} is bounded above by
\[
P_{x_n}(\{nt^0 < \tau_n < \tau_0\} \cap {\mathcal C}_n^c).
\]
The Markov property of $X$,
$ \{  \sigma_1 < \tau_n < \tau_0 \} \subset
 \{   \tau_n < \tau_0 \}$
and \eqref{e:lln}
imply that the last probability is less than
\[
\sum_{x: |x| \le n\delta} P_x(   \tau_n < \tau_0 ) P_{x_n} (X_{nt^0}= x).
\]
For $|x| \le n \delta$, 
the probability $P_x( \tau_n < \tau_0)$ decays exponentially in $n$
\cite[Theorem 2.3]{GlassKou};
then, the above sum goes to $0$. This establishes \eqref{e:toproveconv1} 
and finishes the proof of the proposition.

\end{proof}

\section{Analysis of the limit problem}\label{s:twodim}
In this section and the rest of the paper we will be focusing on the
two tandem queue process and its limit defined in Section \ref{s:intro}.
The analysis of the previous section suggests that we approximate
\[
P_x( \tau_n < \tau_0)
\]
with $W(T_n(x))$ where
\begin{equation}\label{e:simpler}
W(y) \doteq P_y ( \tau < \infty) = {\mathbb E}_y\left[ 1_{\{ \tau < \infty\}}\right].
\end{equation}
The goal of this section is to develop a framework in which we
will derive the following explicit formula for $W$:
\[
W(y) = P_y( \tau < \infty) = 
\rho_2^{y(1)-y(2)} + \frac{\mu_2 - \lambda}{\mu_2 - \mu_1} 
\rho_1^{y(1)-y(2)} \rho_1^{x(2)}
+ \frac{\mu_2-\lambda}{\mu_1 -\mu_2} \rho_2^{y(1)-y(2)} \rho_1^{y(2)},
\]
$ y \in {\mathbb Z}_+^2, y(1) \ge y(2)$;
the proof of this formula is given in as the final result
(Proposition \ref{p:exactformula})  of this section.

It follows from the Markov property of $Y$ that $W$ is a harmonic function of $Y$ (or $Y$-harmonic)
i.e., it satisfies:
\begin{equation}\label{e:linear}
V(y) = {\mathbb E}_y \left[ V(Y_1) \right] = 
\sum_{v \in {\mathcal V} } V(y + \pi_1(y, v)) p(v), y \in B,
\end{equation}
where 
\begin{align}\label{d:pi1}
{\mathcal V} &\doteq \{ (-1,0), (1,1), (0,-1) \}, \notag \\
\pi_1(x,v) &\doteq \begin{cases} v, &\text{ if } x +v  \in {\mathbb Z} \times {\mathbb Z}_+ \\
				  0      , &\text{otherwise.}
\end{cases}
\end{align}
$W(y) = P_y( \tau < \infty) = 1$ for $ y \in \partial B$
implies that $W$ also satisfies the boundary condition
\begin{equation}\label{e:boundary}
V|_{\partial B} = 1.
\end{equation}
A $Y$-harmonic function $h$ is said to be $\partial B$-determined if it is of the form
\[
h(y) = {\mathbb E}\left[f(Y_{\tau}) 1_{\{\tau < \infty\}}\right], y \in {\mathbb Z} \times {\mathbb Z}_+, y(1) \ge y(2).
\]
By its definition, $W$ is $\partial B$-determined.
Then $W$ is the unique $\partial B$-determined solution of (\ref{e:linear},\ref{e:boundary}).

Let $Z$ denote the ordinary unconstrained random walk in ${\mathbb Z}^2$ with
the same increments as $Y$.
The unconstrained version of \eqref{e:linear} is
\begin{equation}\label{e:linear1}
V(z) = {\mathbb E}_z \left[ V(Z_1) \right] = 
\sum_{v \in {\mathcal V} } V(z +v ) p(v), 
~~~ z \in {\mathbb Z}^2.
\end{equation}
A function is said to be a harmonic function of the unconstrained random
walk $Z$ if it satisfies \eqref{e:linear1}.

Our idea for solving
(\ref{e:linear},\ref{e:boundary}) (and hence obtaining a formula for $P_y(\tau< \infty)$)
is this:
\noindent
\begin{enumerate}
\item Construct a class ${\mathcal F}_Y$ 
of ``simple'' harmonic functions for the process $Y$
(i.e., a class of solutions to \eqref{e:linear})
For this
	\begin{enumerate}
	\item
	Construct a class ${\mathcal F}_Z$ of harmonic functions for the
	unconstrained process $Z$  of the form $z\mapsto \beta^{z(1)-z(2)}\alpha^{z(2)}$, $(\beta,\alpha) \in {\mathbb C}^2$,
	\item Use linear combinations of elements of ${\mathcal F}_Z$ to find
		solutions to \eqref{e:linear}.
	\end{enumerate}
	\item 
Represent the boundary condition \eqref{e:boundary}
		by linear combinations of the boundary values of the 
		$\partial B$-determined members of the class ${\mathcal F}_Y$.
\end{enumerate}
The definition of the class ${\mathcal F}_Z$ is given in
\eqref{d:classFZ} and that of ${\mathcal F}_Y$ is given in \eqref{d:classFY}.

A remark about uniqueness:
We have assumed that $X$ is stable; 
this implies that $Y_{\tau \wedge k}$, $k=1,2,3,...$, is unstable and therefore,
the Martin boundary of this process has points at infinity. Then
one cannot expect all harmonic functions of $Y$ to be
$\partial B$-determined and
in particular the system (\ref{e:linear},\ref{e:boundary}) will not
have a unique solution. In particular, the constant function ${\bm 1}(y) = 1$
solves this system, but as we will see below, ${\bm 1}$ is not $\partial B$-determined.
hence, once we find a 
solution of (\ref{e:linear}, \ref{e:boundary}) that we believe
equal to $P_y(\tau < \infty)$, 
we will have to prove that it is $\partial B$-determined.
\subsection{The characteristic polynomial and surface}\label{ss:charpolsurf}
Let us call
\begin{equation}\label{e:hamiltonian}
{\bm p}(\beta,\alpha)
\doteq \ \sum_{ v\in {\mathcal V} } p(v) \beta^{v(1)-v(2)}
\alpha^{v(2)}
=  \lambda \frac{1}{\beta}+ \mu_1 \alpha + \mu_2\frac{\beta}{\alpha}, ~~~~
 (\beta,\alpha) 
\in {\mathbb C}^2,
\end{equation}
the interior {\em characteristic polynomial} of the process $Y$;
\begin{equation}\label{e:chareqZ}
{\bm p}(\beta,\alpha) =1
\end{equation}
the interior {\em characteristic equation} of $Y$ and
\begin{equation}\label{d:DefHcal}
{\mathcal H} \doteq \{ (\beta,\alpha): 
{\bm p}(\beta,\alpha) =1  \}
\end{equation}
the interior {\em characteristic surface} of $Y$.
We borrow the adjective ``characteristic''
from the classical theory of linear ordinary differential equations; the
development below parallels that theory. 
${\bm p}$ is a rational function, not a polynomial, but it obviously becomes polynomial in $\alpha$ $[\beta]$ when multiplied
by $\beta$ [$\alpha$] or a polynomial in $\beta$ and $\alpha$ when multiplied by $\beta\alpha$; these polynomial representations 
are useful when we solve
${\bm p}(\beta,\alpha) = 1$, but the rational representation is simpler, more flexible and natural. For this reason, we use
the rational representation whenever possible, and switch to the polynomial representations when
needed.

Figure \ref{f:charsurftex} depicts the real section of the characteristic
surface of the walk for $\lambda = 0.1$, $\mu_1 = 0.5$ and $\mu_2 =0.4.$
${\mathcal H}$ is an affine algebraic curve
of degree $3$ \cite[Definition 8.1, page 32]{griffiths}.
The characteristic equation ${\bm p} = 1$ becomes  a quadratic 
equation in $\alpha$ when one 
multiplies it by $\alpha$; the discriminant of this quadratic equation is
\[
\Delta(\beta) = \left(\frac{\lambda}{\beta} - 1\right)^2 - 4\mu_1\mu_2\beta.
\]
Therefore, for $\beta \in {\mathbb C}$, $\Delta(\beta) \neq 0$
and $\beta \neq 0$,
points on ${\mathcal H}$ come in conjugate pairs $(\beta,\alpha_1)$ and
$(\beta,\alpha_2)$, satisfying
\begin{equation}\label{e:conjugator0}
\alpha_i = \frac{1}{\alpha_{3-i}}\frac{\mu_2 \beta}{\mu_1},
i \in \{1,2\}.
\end{equation}
These conjugate pairs will be central to the construction
of $Y$-harmonic functions in subsection \ref{ss:twoterms} below.

\begin{figure}[h]
\centering
\hspace{-0.5cm}
\input{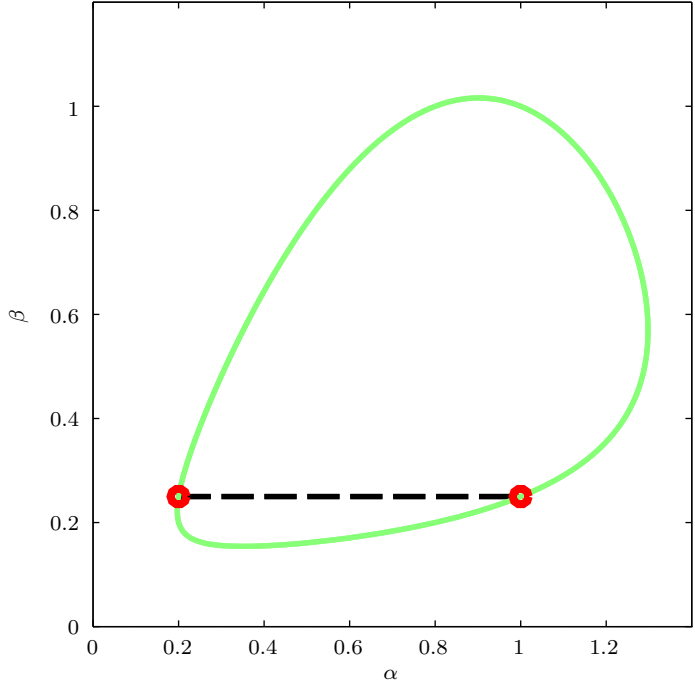}
\label{f:charsurftex}
\caption{The real section of the characteristic surface ${\mathcal H}$
for $\lambda =0.1$, $\mu_1 = 0.5$, $\mu_2 = 0.4$; 
the end points of the dashed line are an example 
of a pair of conjugate points $(\beta,\alpha_1)$
and $(\beta,\alpha_2)$; together
they define the $Y$-harmonic function 
$h_{\beta}(y) = \beta^{y(1)-y(2)} \left(C(\beta,\alpha_2)\alpha_1^{y(2)} - C(\beta,\alpha_1)\alpha_2^{y(2)}\right)$,
see Proposition \ref{p:harmonicYtwoterms2d}.
Each horizontal line intersecting the curve ${\mathcal H}$ twice gives a
pair of conjugate points defining a $Y$-harmonic function}
\end{figure}

Any point on ${\mathcal H}$ defines
a harmonic function of $Z$:
\begin{proposition}\label{p:basicharmonicZ}
For any $(\beta,\alpha) \in {\mathcal H}$,
$z \mapsto \beta^{z(1) - z(2)}\alpha^{z(2)}$, $z \in {\mathbb Z}^2$,
is an harmonic function of $Z$; in particular, it satisfies \eqref{e:linear}
for $y \in {\mathbb Z}^2$, $y(1),y(2) > 0.$
\end{proposition}
\begin{proof}
Condition $Z$ on its first step and use ${\bm p}(\beta,\alpha) = 1.$
\end{proof}

For $(\beta,\alpha) \in {\mathbb C}^2$, define 
\begin{equation}\label{d:bracket}
[(\beta,\alpha),\cdot]: {\mathbb Z}^2 \mapsto {\mathbb C},~~
[(\beta,\alpha),z ]\doteq \beta^{z(1)-z(2)} \alpha^{z(2)}.
\end{equation}
The last proposition gives us the class of harmonic functions 
\begin{equation}\label{d:classFZ}
{\mathcal F}_Z \doteq 
\left\{ [(\beta,\alpha),\cdot],~~
(\beta,\alpha) \in {\mathcal H}
\right\}
\end{equation}
for $Z$.

\subsection{$\log$-linear harmonic functions of $Y$}
Define $B^o \doteq \{y \in {\mathbb Z}_+^2, y(1) >y(2) \}.$
Let us rewrite \eqref{e:linear} separately for the boundary $\partial_2$
and the interior $B^o - \partial_2$:
\begin{align}
V(y) &= \sum_{v \in {\mathcal V} } V(y + v) p(v), y \in B^o -\partial_2,
\label{e:linearint}\\
V(y) &= 
V(y)\mu_2 + 
\sum_{v \in {\mathcal V}, v(2)\neq -1 } V(y + v) p(v), y \in \partial_2 \cap B^o.
\label{e:linearboundary}
\end{align}
Any $ g\in {\mathcal F}_Z$ satisfies \eqref{e:linearint} (because \eqref{e:linearint}
is the restriction of \eqref{e:linear1} to $B^o - \partial_2$);
\eqref{e:linearint} is linear
and
so any finite linear combination of members of ${\mathcal F}_Z$ continues to satisfy
\eqref{e:linearint}. In the next two subsections we will show that
appropriate linear combinations of members of ${\mathcal F}_Z$ 
will also satisfy the
boundary condition \eqref{e:linearboundary} and define harmonic functions
of the constrained process $Y$.

\subsubsection{$Y$-harmonic function defined by a single point on ${\mathcal H}$}\label{ss:asingleterm}
Remember that members of ${\mathcal F}_Z$
are of the form $[(\beta,\alpha),\cdot]:z\rightarrow\beta^{z(1)-z(2)} 
\alpha^{z(2)}$ 
and $(\beta,\alpha) \in {\mathcal H}$; these define harmonic functions for $Z$
and they therefore satisfy \eqref{e:linearint}.
The simples way to construct a $Y$-harmonic function
is to look for $[(\beta,\alpha),\cdot]$ 
which satisfies \eqref{e:linear}, i.e., which satisfies
\eqref{e:linearint} and \eqref{e:linearboundary} at the same time. 
Substituting $[(\beta,\alpha),\cdot]$ in \eqref{e:linearboundary}
 we see that it solves \eqref{e:linearboundary}
if and only if $(\beta,\alpha) \in {\mathcal H}$ also satisfies
\begin{equation}\label{e:chareqYb2}
{\bm p}_{2}(\beta,\alpha) =1
\end{equation}
where
\begin{align}\label{e:hamiltonianb2}
{\bm p}_{2}(\beta,\alpha)
&\doteq 
 \sum_{ v\in {\mathcal V}, v(2)\neq -1 } p(v) \beta^{v(1)-v(2)}
\alpha^{v(2)} + \mu_2 = \lambda \frac{1}{\beta} + \mu_1 \alpha + \mu_2 ;
\end{align}
note 
\begin{equation}\label{e:intermsofPY}
{\bm p}_{2}(\beta,\alpha) = {\bm p}(\beta,\alpha) - \mu_2\left( \frac{\beta}{\alpha} - 1\right).
\end{equation}

Let us call
\eqref{e:chareqYb2} ``the characteristic equation of $Y$ on $\partial_2$''
and ${\bm p}_{2}$ its characteristic polynomial on the same boundary.
Define the boundary characteristic surface of $Y$ for $\partial_2$ as
${\mathcal H}_{2} \doteq \{ (\beta,\alpha) \in {\mathbb C}^2: 
{\bm p}_{2}(\beta,\alpha) = 1 \}$.

For $[(\beta,\alpha),\cdot]$ to  $Y$-harmonic,
$(\beta,\alpha)$
must lie on 
\[
{\mathcal H} \cap{\mathcal H}_{2} = \{ (0,0), (1,1), (\rho_1,\rho_1)\} \subset {\mathbb C}^2;
\]
the third of these points gives us our first nontrivial $Y$-harmonic function:
\begin{proposition}\label{p:singletermharmY}
The function
\begin{equation}\label{e:singletermharmfY}
[(\rho_1,\rho_1),\cdot]: y\mapsto \rho_1^{y(1) - y(2)} \rho_1^{y(1)} 
\end{equation}
is $Y$-harmonic.
\end{proposition}
\begin{proof}
That $[(\rho_1,\rho_1),\cdot]$ satisfies \eqref{e:linearint}
follows from the Markov property of $Y$ and $(\rho_1,\rho_1) \in {\mathcal H}$;
that $[(\rho_1,\rho_1),\cdot]$ satisfies \eqref{e:linearboundary}
follows from the Markov property of $Y$ and  $(\rho_1,\rho_1) \in {\mathcal H}_2$.
\end{proof}

\subsubsection{$Y$-harmonic functions via conjugate points}\label{ss:twoterms}
Define the boundary operator $D_2$ acting on functions on ${\mathbb Z}^2$ and giving
functions on $\partial_2$:
\begin{align*}
&D_2V = g,~~~~ V: {\mathbb Z}^2 \rightarrow {\mathbb C},\\
& g( {\mathrm y},0) \doteq
\left(\mu_2  +  \lambda V( y-1,0) + \mu_1 V(y+1,1) \right)
- V({\mathrm y},0),~~ {\mathrm y} \in {\mathbb Z};
\end{align*}
$D_2$ is the difference between the left and the right sides of \eqref{e:linearboundary}
and gives how much $V$ deviates from being $Y$-harmonic
along the boundary $\partial_2$:
\begin{lemma}\label{l:D2eq0}
 $D_2 V = 0$ 
if and only if $V$ is $Y$-harmonic on $\partial_2$.
\end{lemma}
The proof follows from the definitions involved.
For $(\beta,\alpha ) \in {\mathbb C}^2$ and
$\beta,\alpha \neq 0$
\[
\left[D_2 \left( [({\beta,\alpha}),\cdot] \right)\right]({\mathrm y},0) =  
\left( {\bm p}_2(\beta,\alpha) - 1\right) \beta^{{\mathrm y}}.
\]
where the left side denotes the value of the function 
$D_2 \left( [(\beta,\alpha),\cdot] \right)$ at $({\mathrm y},0)$, 
${\mathrm y} \in {\mathbb Z}$.
By definition, 
${\bm p}(\beta,\alpha)= 1$ for
$(\beta,\alpha) \in {\mathcal H}$;
this, the last display and \eqref{e:intermsofPY}
imply
\begin{equation}\label{e:imagofD2}
\left[D_2 \left( [(\beta,\alpha),\cdot] \right)\right]({\mathrm y},0) =  
\mu_2\left( 1- \frac{\beta}{\alpha} \right)
\beta^{{\mathrm y}}
\end{equation}
if $(\beta,\alpha) \in {\mathcal H}$.
One can write the function $({\mathrm y},0) \mapsto \beta^{{\mathrm y}}$
as $[(\beta,\alpha),\cdot] |_{\partial_2} = [(\beta,1),\cdot] |_{\partial_2}$;
in addition, define
\begin{equation}\label{e:defC2d}
C(\beta,\alpha) \doteq \mu_2\left(1 -  \frac{\beta}{\alpha} \right).
\end{equation}
With these, rewrite \eqref{e:imagofD2}
as
\begin{equation}\label{e:imageofD2short}
D_2 \left( [(\beta,\alpha),\cdot] \right)  = 
C(\beta,\alpha) [(\beta,1),\cdot]|_{\partial_2}.
\end{equation}
The key observation here
is this: $D_2\left( [(\beta,\alpha),\cdot]\right)$ is a constant
multiple of $[(\beta,1),\cdot] |_{\partial_2}$.
This and the linearity of $D_2$ imply that for
\begin{equation}\label{e:conjugacy}
\alpha_1 \neq \alpha_2, ~~(\beta,\alpha_1), (\beta,\alpha_2)  \in {\mathcal H},
\end{equation}
i.e., when $(\beta,\alpha_1)$ and $(\beta,\alpha_2)$ are conjugate points on ${\mathcal H}$,
$[(\beta,\alpha_1),\cdot]$ and $[(\beta,\alpha_2),\cdot]$
can be linearly combined to cancel out
each other's value under $D_2$. 
The next proposition uses these
conjugate pairs and the above argument to find new $Y$-harmonic functions:
\begin{proposition}\label{p:harmonicYtwoterms2d}
Assume $\beta \in {\mathbb C}$, $\beta \neq 0$ satisfies 
$\Delta(\beta)\neq 0$. Then
\begin{equation}\label{e:harmonicY2d}
h_\beta \doteq  C(\beta,\alpha_2) [(\beta,\alpha_1),\cdot]
- C(\beta,\alpha_1) [(\beta,\alpha_2),\cdot]
\end{equation}
is $Y$-harmonic.
\end{proposition}
\begin{proof}
By assumption
$(\beta,\alpha_1), (\beta,\alpha_2)$ are both on ${\mathcal H}$ and
therefore $[(\beta,\alpha_1),\cdot]$ and $[(\beta,\alpha_2),\cdot]$ are
 harmonic functions of $Z$ and in particular, they both satisfy
\eqref{e:linearint}. Then their linear combination $h_\beta$ also
satisfies \eqref{e:linearint}, because \eqref{e:linearint}  is linear in $V$.
It remains to show that $h_\beta$ solves \eqref{e:linearboundary} as well.
$\beta \neq 0$ implies $\alpha_1,\alpha_2 \neq 0,1$. Then 
\eqref{e:imageofD2short} implies
\begin{align*}
D_2(h_\beta) &= 
 C(\beta,\alpha_2) 
D_2( [(\beta,\alpha_1),\cdot]) -  C(\beta,\alpha_1) D_2([\beta,\alpha_2,\cdot])\\
&= 
C(\beta,\alpha_2) 
C(\beta,\alpha_1)  [(\beta,1),\cdot]|_{\partial_2}
-
C(\beta,\alpha_1)
C(\beta,\alpha_2)  [(\beta,1),\cdot]|_{\partial_2}\\
& = 0
\end{align*}
and
Lemma \ref{l:D2eq0} implies that $h_\beta$ satisfies \eqref{e:linearboundary}.
\end{proof}

The function $W(y)=P_y(\tau < \infty)$  takes the value $1$ on $\partial B$. For
this reason,
the conjugate pair on ${\mathcal H}$ that is most relevant to the computation of
$P_y(\tau < \infty)$ consists of $(\rho_2,1)$ and $(\rho_2,\rho_1)$; 
this pair is shown in Figure \ref{f:charsurftex}. $h_{\rho_2}$, the $Y$-harmonic function
defined by this pair, equals
\begin{align*}
h_{\rho_2}(y) &= C(\rho_2,\rho_1) [(\rho_2,1),y]  - C(\rho_2,1)[(\rho_2,\rho_1),y]
\intertext{which, by definitions \eqref{d:bracket} and \eqref{e:defC2d}, equals}
&= (\mu_2 - \mu_1) \rho_2^{y(1)-y(2)}  - (\mu_2 -\lambda) \rho_2^{y(1)-y(2)} \rho_1^{y(2)}.
\end{align*}
Note that the first term in the definition \eqref{d:Wstar} of $W^*$ equals $\frac{1}{\mu_2-\mu_1} h_{\rho_2}.$

With Proposition \ref{p:harmonicYtwoterms2d} 
 we define our basic class of harmonic functions of $Y$:
\begin{equation}\label{d:classFY}
{\mathcal F}_Y \doteq \{ h_\beta, \beta \neq 0, \Delta(\beta) \neq 0 \}.
\end{equation}

Members of ${\mathcal F}_Y$ consist of linear combinations of $\log$-linear functions;
with a slight abuse of language, we will also refer to such functions as
$\log$-linear.

\begin{remark}{\em
For the purposes of computing $P_y(\tau < \infty)$ for the tandem network case
treated in the present paper a single member of ${\mathcal H}_Y$ will suffice,
i.e., $h_{\rho_2}$, see Proposition \ref{p:exactformula} below. 
But ${\mathcal H}_Y$ is a whole family of simple to compute $Y$-harmonic functions and
they can be used to approximate other expectations or even $P_y(\tau < \infty)$ when
the underlying network is not tandem, see Remark \ref{r:r2} below.
}
\end{remark}

\subsection{Graph representation of $\log$-linear harmonic functions of $Y$}
\label{ss:graphs}
Figure \ref{f:graphsharmtex}
gives a graph representation of the harmonic functions developed
in the last subsection.

\begin{figure}[h]
\begin{center}
\scalebox{0.8}{
\centerline{\input{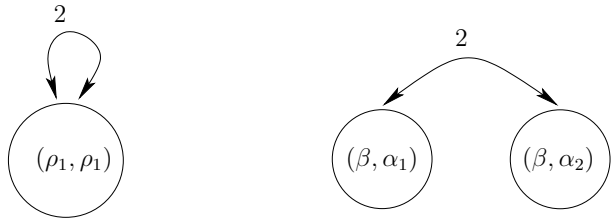}}}
\end{center}

\vspace{-0.65cm}
\caption{\hspace{0.25cm}Graph representation of $Y$-harmonic functions
constructed in Propositions \ref{p:singletermharmY} and \ref{p:harmonicYtwoterms2d}\label{f:graphsharmtex}}
\end{figure}

Each node in this figure represents a member of ${\mathcal F}_Z$. 
The edges represent the boundary conditions; in this case there is only
one, \eqref{e:linearboundary} of $\partial_2$, and the edge label ``$2$''
 refers to $\partial_2$.
A self connected vertex represents a member of ${\mathcal F}_Z$ that also
satisfies the $\partial_2$ boundary condition \eqref{e:linearboundary}, i.e.,
$z \rightarrow \beta(r_1)^{z(1)} r_1^{z(1)-z(2)}$ 
of Proposition \ref{p:singletermharmY};
the graph on the left represents exactly this function. The ``$2$'' labeled
edge on the right represents the conjugacy relation
\eqref{e:conjugator0} between $\alpha_1$ and $\alpha_2$, which allows these functions
to be linearly combined to satisfy the harmonicity condition of $Y$ on $\partial_2$.

{\em We call the graphs shown in Figure \ref{f:graphsharmtex} and the system 
of characteristic equations
they represent a {\em harmonic system}.} 
One can define harmonic systems for $d$ dimensional constrained
random walks as well (see
\cite[Section 5]{sezer2015exit}; these systems and their solutions
play a key role in the generalization of the analysis of this section to higher dimensions.

\subsection{$\partial B$-determined harmonic functions of $Y$}
\label{ss:partialBdet}
In the subsections \ref{ss:asingleterm} and \ref{ss:twoterms} above we have constructed classes of
$Y$-harmonic functions. For the purposes of computing $W(y) = P_y(\tau < \infty)$, $y \in B$, we need
$\partial B$-determined $Y$-harmonic functions. Proposition \ref{p:balayagesimple}
derives simple conditions that allow one check when a member of
${\mathcal F}_Y$ is $\partial B$ determined. In this, the following
fact will be useful.
\begin{lemma}\label{p:cantwanderforever}
Define
\begin{equation}\label{e:defofSigman}
\zeta_n \doteq \inf\left\{k: Y_k(1) = Y_k(2) + n\right\}.
\end{equation}
For $y \in {\mathbb Z}_+^2$, $0 \le y(1) - y(2) \le n$,
\begin{equation}\label{e:toprove0}
P_y( \zeta_n \wedge \zeta_0 = \infty) = 0.
\end{equation}
\end{lemma}
\begin{proof}
The proof follows from the fact that, when in  $C = \{y \in {\mathbb Z}^2_+, y(2) \le y(1) \le y(2) + n\}$
the process $Y$ hits $\partial C = \{y \in {\mathbb Z}_+^2 :y(1)-y(2) = n \text{ or } (1)=y(2)\}$ in
at most $n$ steps with probability greater than $\lambda^n$. For a detailed version of this
argument we refer the reader to \cite[Proof of Proposition 2.2]{sezer2015exit}.
\end{proof}

\begin{proposition}\label{p:balayagesimple}
Let $\alpha_1$, $\alpha_2$ and $\beta$ be as in Proposition
\ref{p:harmonicYtwoterms2d}. 
If 
\begin{equation}\label{e:simpleconditions}
|\beta| < 1,~~~|\alpha_1|,|\alpha_2|  \le 1
\end{equation}
then $h_\beta$ of \eqref{e:harmonicY2d} is $\partial B$-determined.
\end{proposition}
\begin{proof}
By Proposition \ref{p:harmonicYtwoterms2d}  $h_\beta$ is 
$Y$-harmonic;
\eqref{e:simpleconditions} and its definition \eqref{e:harmonicY2d}
imply that $h_\beta$ is also bounded on $B^o$. 
Then
$M_k  =  h_\beta(Y_{\tau \wedge \zeta_n \wedge k})$ is a bounded martingale. 
This, Proposition \ref{p:cantwanderforever} and the optional
sampling theorem imply
\begin{align}\label{e:resofopsamp}
h_\beta(y) &= 
{\mathbb E}_y
\left[
h_\beta(Y_{\tau}) 1_{\{\tau < \zeta_n\}} \right]
+ {\mathbb E}_y\left[
h_\beta(Y_{\zeta_n}) 
1_{\{ \zeta_n \le \tau \}} \right], y \in B^o.
\end{align}
$Y_{\zeta_n}(1) = n$ for $\tau > \zeta_n.$ This and \eqref{e:simpleconditions}
imply 
\[
\lim_{n\rightarrow \infty}
 {\mathbb E}_y\left[
h_\beta(Y_{\zeta_n}) 
1_{\{ \zeta_n \le \tau \}} \right]
\le \lim_{n\rightarrow \infty} \beta^n = 0.
\]
This, $\lim_n \zeta_n = \infty$ and letting $n\rightarrow \infty$
in \eqref{e:resofopsamp}
give
\[
h_\beta(y) ={\mathbb E}_y \left[ h_\beta(Y_{\tau}) 1_{\{\tau <\infty \}} 
\right],
\]
i.e, $h_\beta$ is $\partial B$-determined.
\end{proof}
In addition, we have:
\begin{proposition}\label{p:singletermBdet}
The $Y$-harmonic function $[(\rho_1,\rho_1),\cdot]$ of Proposition
\ref{p:singletermharmY}
 is  $\partial B$-determined.
\end{proposition}
\begin{proof} The proof is identical to that of
Proposition \ref{p:balayagesimple} and follows from 
$0 \le [(\rho_1,\rho_1),y] \le 1$ for $y \in B$ and the $Y$-harmonicity
of $[(\rho_1,\rho_1),\cdot].$
\end{proof}

Proposition \ref{p:balayagesimple} rests on the condition
 \eqref{e:simpleconditions}; we refer the reader to \cite[Section 4]{sezer2015exit}, in particular
Proposition 4.13 that derives conditions under which \eqref{e:simpleconditions} hold in the
context of general two node Jackson networks. For the purposes of computing $P_y(\tau < \infty)$, we only
need to consider the point $(\rho_1,\rho_1)$ and the conjugate pair $(\rho_2,1)$ and $(\rho_2,\rho_1)$;
it is trivial to check the conditions \eqref{e:simpleconditions} for these points. This gives us
the main result of this section:
\begin{proposition}\label{p:exactformula}
Under the stability assumption $\lambda< \mu_1,\mu_2$, $h_{\rho_2}$ is $\partial B$-determined
and we have
\[
P_y( \tau < \infty) = W^*(y)= \frac{1}{C(\rho_2,\rho_1)} h_{\rho_2}(y) + \frac{C(\rho_2,1)}{C(\rho_2,\rho_1)} [(\rho_1,\rho_1),y].
\]
$y \in B.$
\end{proposition}
The definitions \eqref{d:bracket} and \eqref{e:defC2d} give us the following expanded formula
for $W^*$:
\begin{align*}
W^*(y)&= \frac{1}{C(\rho_2,\rho_1)} h_{\rho_2}(y) + \frac{C(\rho_2,1)}{C(\rho_2,\rho_1)} [(\rho_1,\rho_1),y]\\
&=\left( \rho_2^{y(1)-y(2)} 
+ \frac{\mu_2-\lambda}{\mu_1 -\mu_2} \rho_2^{y(1)-y(2)} \rho_1^{y(2)} \right)
+ \frac{\mu_2 - \lambda}{\mu_2 - \mu_1} 
\rho_1^{y(1)-y(2)} \rho_1^{y(2)},
\end{align*}
which is the one given in \eqref{d:Wstar}, in the introduction.

\begin{proof}
The conjugate points on ${\mathcal H}$ 
for $\beta=\rho_2$ are
$(\rho_2,1)$ and $(\rho_2,\rho_1)$; the stability assumption
$\lambda < \mu_1,\mu_2$ implies that both of these points satisfy \eqref{e:simpleconditions}.
It follows from Propositions \ref{p:harmonicYtwoterms2d} and \ref{p:balayagesimple}  that $h_{\rho_2}$ is a $\partial B$ determined
$Y$-harmonic function; similarly, it follows from Propositions \ref{p:singletermharmY} and \ref{p:singletermBdet} that
$[(\rho_1,\rho_1),\cdot]$ is a $\partial B$-determined $Y$-harmonic function. It follows that their linear combination
$W^*$ is also $\partial B$-determined and $Y$-harmonic, i.e.,
\begin{align*}
W^*(y) &= {\mathbb E}_y[ 1_{\{\tau < \infty\}} W^*(Y_\tau)]
\intertext{But $W^*(y) = 1$ on $\partial B$; therefore,}
&= P_y(\tau < \infty).
\end{align*}
\end{proof}

\begin{remark}\label{r:balayage}
{\em
The Balayage operator ${\mathcal T}$ (see \cite[page 25]{revuz1984markov}) for the set $\partial B$
is the operator mapping a function $f$ on $\partial B$ to the $Y$-harmonic function $g$ on $B$, defined as follows:
\[
{\mathcal T}: f \rightarrow g,  
g(x) = {\mathbb E}_x\left[ f\left(X_{\tau}\right) 1_{\{ \tau < \infty \}} \right].
\]
Therefore, by definition, a $Y$-harmonic function $h$ is $\partial B$-determined, if and only if it is the
image of some function under the Balayage operator ${\mathcal T}.$ Computing $P_y(\tau < \infty)$ amounts to
computing the image of the constant function $1$ on $\partial B$ under the Balayage operator.
What Propositions \ref{p:singletermharmY}, \ref{p:harmonicYtwoterms2d}, \ref{p:balayagesimple}, and \ref{p:singletermBdet} 
do is they give us a collection of 
{\em basis functions} for which the Balayage operator ${\mathcal T}$
is very simple to compute; 
these functions play the same role for the current problem
as the one which exponential functions do
in the solution of linear ordinary differential equations
or the trigonometric functions in the solution
of the heat and the Laplace equations.
Let us rewrite Proposition \ref{p:balayagesimple} more explicitly.
Suppose $\alpha_1$, $\alpha_2$ and $\beta$ 
are as in Proposition \ref{p:balayagesimple};
recall that
\[
h_\beta(y) = 
\beta ^{y(1)-y(2)}\left(
 C(\beta,\alpha_2)\alpha_1^{y(2)} - 
 C(\beta,\alpha_1)\alpha_2^{y(2)}\right), y \in {\mathbb Z}^2.
\]
Then, Proposition \ref{p:balayagesimple} says
\begin{equation}\label{e:explicitBalayage}
{\mathbb E}_y\left[ 
h_\beta(Y_\tau)1_{\{\tau < \infty\} }\right] = h_\beta(y).
\text{ i.e., }
{\mathcal T}(h_\beta|_{\partial B}) = h_\beta.
\end{equation},
}
\end{remark}
\begin{remark}\label{r:r2}{\em
In this article we are interested in the computation of $P_y(\tau < \infty) = {\mathbb E}_y[ 1_{\tau < \infty} ].$
More generally we may be interested in computing $g(y)={\mathbb E}_y[f(Y_\tau) 1_{\{\tau < \infty\}}]$ for some function $f$.
To approximate this expectation, one can proceed as follows. First, 
approximate $f$ with a finite superposition of the form
\[
f^* = \sum_{i=1}^K w_i f_i|_{\partial B}, 
\]
where $w_i \in {\mathbb C}$ and $f_i \in {\mathcal H}_{\mathcal Y}$,
i.e., a $Y$-harmonic function of the form
\[
f_i = C(\beta_i,\alpha_i^*)[(\beta_i,\alpha_i),\cdot] - C(\beta_,\alpha_i)[(\beta_i,\alpha_i^*),\cdot],
\]
and $|\beta_i|,|\alpha_i|,|\alpha^*_i| < 1$;
then by \eqref{e:explicitBalayage}
\[
{\mathbb E}_y[f^*(Y_{\tau}) 1_{\{\tau < \infty\}}]  = \sum_{i=1}^K w_i f_i(y)
\]
would give an approximation of  ${\mathbb E}_y[f(Y_\tau) 1_{\{\tau < \infty\}}]$ for $y \in B.$
The error made in this approximation will be bounded by $\max_{y\in \partial B} | f^*(y) - f(y)|.$
}
\end{remark}

\section{Convergence - initial condition set for $X$}\label{s:convergence2}
The convergence argument
of Section \ref{s:convergence1}
used an initial point for the $Y$ pricess.
The goal of this section is to provide a convergence argument starting
from an initial position specified for the $X$ process
as $X(0) = \lfloor n x \rfloor$ for  a fixed
$x \in {\mathbb R}_+^2$ with 
$x(1)+x(2) < 1$, as is done in LD analysis. We will show that the relative
error 
\[
\frac{|P_{x_n}(\tau_n < \tau_0) - P_{T(x_n)}(\tau < \infty)|}{P_{x_n}(\tau_n < \tau_0)}
\]
decays {\em exponentially} in $n$; see Proposition \ref{t:guzel} below.

For the present analysis we will also use the limit process $Y$ expressed
in the original coordinates of the $X$ process, which is $\bar{X} \doteq T_n(Y)$.
$\bar{X}$ is the same process as $X$ except that it
is constrained only at the boundary
$\partial_2.$
\begin{align*}
X_{k+1} &= X_k+ \pi(X_k,I_k)\\
\bar{X}_{k+1} &= \bar{X}_k + \pi_1(\bar{X}_k,I_k),
\end{align*}
where $\pi_1$ is as in \eqref{d:pi1}.
We will assume that $X$ and $\bar{X}$ start from the same initial position
\[
X_0= \bar{X}_0
\]
and whenever we specify an initial position below it will be for both processes.

As before,
$\tau_n =  \inf\{k: X_1(k) + X_2(k) = \partial A_n \}$ and
$\tau =\inf\{k: Y_k \in \partial B\}$; define $\partial\bar{A}_n \doteq \{x \in {\mathbb Z} \times {\mathbb Z}_+:
x(1) + x(2) = n \}$; 
By definition, $\bar{X}$ hits $\partial \bar{A}_n$ exactly when $Y$ hits $\partial B$;
therefore, $\tau = \bar{\tau}_n \doteq \inf\{k: \bar{X}_k \in \partial \bar{A}_n\}$,
and $P_{T(x_n)}(\tau < \infty) = P_{x_n}(\bar{\tau}_n < \infty)$.

\begin{proposition}\label{t:guzel}
For $x \in {\mathbb R}_+^2$, $0 < x(1) + x(2) < 1, x(1) > 0$ set
$ x_n \doteq \lfloor nx \rfloor$.
Then
\begin{equation}\label{e:relativeerror}
\frac{
|P_{x_n}(\tau_n < \tau_0) - P_{T(x_n)}( \tau < \infty )|}
	{P_{x_n}(\tau_n < \tau_0)}
=
\frac{|P_{x_n}(\tau_n < \tau_0) - P_{x_n}( \bar{\tau}_n < \infty )|}
	{P_{x_n}(\tau_n < \tau_0)}
\end{equation}
decays exponentially in $n$.
\end{proposition}
The proof will require several supporting results on 
$\sigma_1 =\inf\{ k: X_k \in \partial_1 \}$ and
\begin{align*}
\sigma_{1,2} &\doteq \inf \{k: k \ge  \sigma_1, X_k \in \partial_2 \},\\
\bar{\sigma}_{1,2} &\doteq \inf \{k: k \ge \sigma_1, \bar{X}_k(1) = -\bar{X}_k(2) \}.
\end{align*}
\begin{proposition}\label{p:equalityofsums}
\begin{equation}\label{e:equalityofsums}
X_k(1)+ X_k(2) = \bar{X}_k(1) + \bar{X}_k(2)
\end{equation}
for $k \le \sigma_{1,2}.$
\end{proposition}
\begin{proof}
\begin{equation}\label{e:equalityatsigma1}
X_k = \bar{X}_k
\end{equation}
for $k \le \sigma_1$ implies \eqref{e:equalityofsums} for $k  \le \sigma_1$.
If $\sigma_1 = \sigma_{1,2}$ then we are done. Otherwise 
$X_{\sigma_1}(2) = \bar{X}_{\sigma_1}(2) > 0$ and $X_k(2) > 0$ for 
$\sigma_1 < k < \sigma_{1,2}$;  
let $\sigma_1 = \nu_1 < \nu_2 < \cdots <\nu_K < \sigma_{1,2}$ 
be the times when
$X$ hits $\partial_1$ before hitting $\partial_2.$ The definitions of $\bar{X}$
and $X$ imply that these are the only times
when the increments of $X$ and $\bar{X}$ differ: $X_{\nu_j+1} - X_{\nu_j}=0$ 
and $\bar{X}_{\nu_j+1} - \bar{X}(\nu_j) = (-1,1)$
if
$I_{\nu_j}= (-1,1)$; otherwise both differences equal $I_{\nu_j}$. This and
\eqref{e:equalityatsigma1} imply
\begin{equation}\label{e:invariantinc}
X_k - \bar{X}_k = \varsigma_k \cdot ( -1,1)
\end{equation}
for $k \le \sigma_{1,2}$
where 
\[
\varsigma_k \doteq \sum_{j = 1}^K 1_{\{ \nu_j \le k\} } 1_{\left\{ I_{\nu_j} = (-1,1) 
\right\}}
\]
and $\cdot$ denotes scalar multiplication.
Summing the components of both sides of \eqref{e:invariantinc}
gives \eqref{e:equalityofsums}.
\end{proof}

Define
\[
\Gamma_n \doteq \{ \sigma_1 < \sigma_{1,2} < \tau_n < \tau_0\}.
\]
$\Gamma_n$ is one particular way for $\{\tau_n < \tau_0\}$ to occur.
In the next proposition we find an upperbound on its probability 
in terms of
\[
\gamma \doteq -(\log(\rho_1) \vee \log( \rho_2)).
\]

\begin{proposition}\label{t:thm1}
For any $\epsilon > 0$
there is  $N > 0 $ such that if $n > N$
\begin{equation}\label{e:bound1}
P_{x_n} ( \Gamma_n ) \le e^{-n (\gamma-\epsilon)},
\end{equation}
where $x_n = \lfloor nx\rfloor$ and $x \in{\mathbb R}_+^2$, $x(1) + x(2) < 1.$
\end{proposition}
The proof will use the following definitions. Let $v_0 =(0,1)$, $v_1 = (-1,1)$ ,
$v_2 = (0,-1)$, $p_X(v_0) = \lambda$, $p_X(v_1) = \mu_1$, $p_X(v_2) = \mu_2$ and
\begin{equation}\label{e:defHs}
H_a(q)  \doteq
-\log\left( \sum_{i \in \{0,1,2\} -a } p_X(v_i) e^{-\langle v_i, q \rangle} 
+ \sum_{\{ i \in a\}}p_X(v_i) \right), ~~ a \subset \{1,2\},
\end{equation}
where $\langle \cdot,\cdot\rangle$ denotes the inner product in ${\mathbb R}^2.$
For $x \in {\mathbb R}_+^2$, 
set
\[
{\bm b}(x) \doteq \{i: x(i) = 0 \}.
\]
We will write $H$ rather than $H_\emptyset$. 

Let us show the gradient operator on smooth functions on ${\mathbb R}^2$ 
with $\nabla$.
The works \cite{thesis,DSW} use a smooth subsolution of 
\begin{equation}\label{e:HJB0}
H_{{\bm b}(x)}(\nabla V(x)) = 0
\end{equation}
to find a lowerbound on the decay rate of the second moment of IS estimators
for the probability $P_{x_n}(\tau_n  < \tau_0)$.
$V$ is said to be a subsolution of \eqref{e:HJB0} 
if $H_{{\bm b}(x)}(\nabla V(x)) \ge 0$.
The event $\Gamma_n$
consists of three stages: the process first hits $\partial_1$, then $\partial_2$
and then finally hits $\partial A_n$ without hitting $0$. 
To handle this, we will use a function
$(s,x)\rightarrow V(s,x)$, 
with two variables;
for the $x$ variable we will substitute 
the scaled position of the $X$ process,
and the discrete variable $s \in \{0,1,2\}$  is 
for keeping track of which of the above three stages
the process is in; $V$ will be a subsolution in the $x$ variable 
and continuous in $(s,x)$ (when $(s,x)$ is thought of as a point on the
manifold ${\mathcal M}$ 
consisting of three copies of ${\mathbb R}_+^2$ (one for each stage); the zeroth glued to the
first along $\partial_1$ and the first to the second along $\partial_2$)
 and therefore one can think of
$V$ as three subsolutions (one for each stage) glued together along the boundaries
of the state space of $X$ where transitions between the stages occur. 
We will call a function $(s,x) \rightarrow V(s,x)$ with the above properties
a subsolution of \eqref{e:HJB0} on the manifold ${\mathcal M}.$

Define
\begin{equation}\label{e:subsolpa}
\tilde{V}_i^{\varepsilon}(x) \doteq \langle {\boldsymbol r}_i, x\rangle + 2\gamma - (3-i)\varepsilon,~~~
\tilde{V}^{\varepsilon,j} \doteq \bigwedge_{i=0}^j \tilde{V}_i^{\varepsilon},
\end{equation}
where
\[
{\boldsymbol r}_0 \doteq(0,0), {\boldsymbol r}_1 = -\gamma(1,0), {\boldsymbol r}_2 \doteq -\gamma(1,1).
\]
The subsolution for stage $j$ will be a smoothed version of $\tilde{V}^{\varepsilon,j}$;
As in \cite{thesis, DSW}, 
we will need to vary $\varepsilon$ with $n$ in the convergence argument; for this
reason, $\varepsilon$ will appear as the third parameter of the constructed subsolution.
The details are as follows.

The subsolution for the zeroth stage is $\tilde{V}^{0,\varepsilon}$:
$V(0,x,\varepsilon) \doteq \gamma - 3\varepsilon$,
$\nabla V(0,\cdot) =0$ and it trivially satisfies \eqref{e:HJB0} and
is therefore a subsolution.

Define the smoothing kernel
\[
\eta_{\delta}(x) \doteq \frac{1}{\delta^2 M} \eta(x/\delta),~~\\
\eta(x) \doteq 1_{\{|x| \le 1\} } (|x|^2 - 1), M \doteq \int_{ {\mathbb R}^2} \eta(x) dx\\
\]

To construct the subsolution for the first and the second stages we will mollify
$\tilde{V}^{j,\varepsilon}$, $j=1,2$, with $\eta$:
\begin{equation}\label{e:subsolpas}
V(j,x,\varepsilon) \doteq \int_{{\mathbb R}^2} \tilde{V}^{j,\varepsilon}(y)
\eta_{c_2 \varepsilon} (x-y) dy,
\end{equation}
and $c_2$ is chosen so that 
\begin{equation}\label{e:V12nearonp2}
V(1,x,\varepsilon) = V(2,x,\varepsilon) 
\end{equation}
for $x \in
\partial_2$ and 
\begin{equation}\label{e:V01nearonp1}
V(1,x,\varepsilon) = V(0,x,\varepsilon) 
\end{equation}
for $x \in \partial_1$
(this is possible since $V(j,\varepsilon,x) \rightarrow \tilde{V}^{j,\varepsilon}$ 
as $c_2 \rightarrow 0$ and all of the involved functions are affine; see
\cite[page 38]{thesis} on how to compute $c_2$ explicitly).
That
$V(j,\cdot,\varepsilon)$, $j=1,2$ are subsolutions follow the concavity
of $H_a$ and the choices of the gradients ${\boldsymbol r}_i$; for details
we refer the reader to
\cite[Lemma 2.3.2]{thesis}; a direct computation gives
\begin{equation}\label{e:Boundon2ndderivative}
\left| \frac{ \partial^2 V(j,\cdot,\varepsilon)}{\partial x_i \partial x_j} \right| 
\le \frac{c_3}{\varepsilon},
\end{equation}
$j=1,2$, for a constant $ c_3 > 0$ (again, the proof of \cite[Lemma 2.3.2]{thesis}
gives the details of this computation).

The construction above implies
\begin{equation}\label{e:valueofVonpartialn}
V(2,x,\varepsilon ) < 0, x \in \{x: x(1) + x(2)= 1\}.
\end{equation}

Now on to the proof of Proposition \ref{t:thm1}.
\begin{proof}
$V(0,\cdot,\varepsilon)$ maps to a constant and thus
\begin{equation}\label{e:exactreplacement}
\langle \nabla W(x), v_i \rangle = W(x+v_i) - W(x) 
\end{equation}
if $W=V(0,\cdot,\varepsilon)$.
For $W= V(j,\cdot,\varepsilon)$, $j=1,2$, 
Taylor's formula and \eqref{e:Boundon2ndderivative} give
\begin{equation}\label{e:approxforV2}
\left| 
 \left\langle \nabla W(x), \frac{1}{n}v_i \right\rangle 
-\left( W\left(x+\frac{1}{n} v_i\right) - W(x) \right) \right| \le \frac{c_3}{n\varepsilon}.
\end{equation}
We will allow $\varepsilon$ to depend on $n$ so that
$\varepsilon_n \rightarrow 0$ and
$n\varepsilon_n  \rightarrow \infty.$
Define
$S_k = 1_{\{\sigma_1 > k \}} + 1_{\{ \sigma_{1,2} > k \}}$,
$M_0 \doteq 1$ and
{\small
\begin{align*}
M_{k+1} \doteq M_k
 \exp\left( -n 
\left( V \left(S_{k+1},\frac{X_{k+1}}{n}, \varepsilon_n\right)-V\left(S_k,
\frac{X_k}{n},\varepsilon_n\right)\right)
 -1_{\{n > \sigma_{1}\}} \frac{c_3}{n\varepsilon_n}\right)
\end{align*}
}
That $V(j,\cdot,\varepsilon_n)$, $j=0,1,2$ are subsolutions of \eqref{e:HJB0}, the relations
\eqref{e:exactreplacement}, \eqref{e:approxforV2}
\eqref{e:V01nearonp1} and \eqref{e:V12nearonp2}
imply that 
$M$ is a supermartingale;
\eqref{e:approxforV2} 
and
\eqref{e:exactreplacement} 
allow us to replace gradients
in \eqref{e:HJB0} and \eqref{e:defHs} with finite differences and 
\eqref{e:V12nearonp2} 
and 
\eqref{e:V01nearonp1} 
preserve the supermartingale
property of $M$ as $S$ passes from $0$ to $1$ and from $1$ to $2$.
This and $M \ge 0 $ imply (see \cite[Theorem 7.6]{MR1609153})
{\small
\[
{\mathbb E}_{x_n} 
\left[
\prod_{k=1}^{\tau_{0,n}} 
 \exp\left( -n 
\left( V \left(S_{k+1},\frac{X_{k+1}}{n},\varepsilon_n\right)-
V\left(S_k,\frac{X_k}{n},\varepsilon_n\right)\right) 
-1_{\{n > \sigma_{1}\}} \frac{c_3}{n\varepsilon_n}\right)
\right] \le 1,
\]
}
where $\tau_{0,n} \doteq \tau_n \wedge \tau_0.$
Restrict the expectation on the left to  $1_{\Gamma_n}$
and replace $1_{\{n > \sigma_{1}\} }$ with $1$ to make the expectation
smaller:
{\small
\begin{align*}
&{\mathbb E}_{x_n} 
\left[
1_{\Gamma_n}
e^{-\frac{c_3}{n\varepsilon_n} \tau_{0,n}}  
 \exp\left( -n 
\sum_{k=1}^{\tau_{0,n}} 
 V \left(S_{k+1},\frac{X_{k+1}}{n},\varepsilon_n\right)
-V\left(S_k,\frac{X_k}{n},\varepsilon_n\right)
\right)  \right] \le 1.
\end{align*}
}
Over $\Gamma_n$, $X$ first hits $\partial_1$ and then $\partial_{2}$
and finally $\partial A_n$. Furthermore, 
the sum inside the expectation is telescoping
across this whole trajectory; these imply that 
the last inequality reduces to
\[
{\mathbb E}_{x_n} 
\left[
1_{\Gamma_n}
e^{-\frac{c_3}{n\varepsilon_n} \tau_{0,n}} 
\exp( -n( V(2,X_{\tau_{0,n}},\varepsilon_n) - V(0,X_0,\varepsilon_n)) 
\right] \le 1.
\]
$\tau_{0,n} = \tau_n$ on $\Gamma_n$ and therefore on the same set
$X_{\tau_{0,n}} \in \partial_n$. This, $V(0,\cdot,\epsilon_n)= \gamma-3\epsilon_n$,
\eqref{e:valueofVonpartialn}
and the previous inequality give
\begin{equation}\label{e:basic}
{\mathbb E}_{x_n} 
\left[
1_{\Gamma_n}
e^{-\frac{c_3}{n\varepsilon_n} \tau_{0,n}} 
\right] \le e^{-n( \gamma - 3 \varepsilon_n)}.
\end{equation}
Now suppose that the statement of Theorem \ref{t:thm1} is not true,
i.e., there exists $\epsilon > 0$  and a sequence $n_k$ such that
\begin{equation}\label{e:contra}
P_{x_{n_k}}(\Gamma_{n_k}) > e^{-n_k (\gamma -\epsilon)}
\end{equation}
for all $k$.  Let us pass to this subsequence and drop the subscript
$k$. 
\cite[Theorem A.1.1]{thesis} implies that 
one can choose $ c_4 > 0$ so that $P( \tau_{0,n} > n c_4 ) \le e^{-n(\gamma + 1) }$
for $n$ large.
Then
\begin{align*}
{\mathbb E}_{x_n} 
\left[
1_{\Gamma_n}
e^{-\frac{c_3}{n\varepsilon_n} \tau_{0,n}} 
\right] 
&\ge
{\mathbb E}_{x_n} 
\left[
1_{\Gamma_n}
e^{-\frac{c_3}{n\varepsilon_n} \tau_{0,n}} 
 1_{\{\tau_{0,n} \le n c_4\}}
\right] 
\\
&\ge
e^{-\frac{c_4 c_3}{n\varepsilon_n} n }
{\mathbb E}_{x_n} 
\left[
1_{\Gamma_n}
 1_{\{\tau_{0,n} \le n c_4\}}
\right] 
\intertext{
$P(E_1 \cap E_2) \ge 
P(E_1) - P(E_2^c)$
for any two events $E_1$ and $E_2$; this and the previous line imply}
&\ge
e^{\frac{-c_3 c_4}{n\varepsilon_n} n } \left( P_{x_n}(\Gamma_n) - P_{x_n}( \tau_{0,n} > nc_4)
\right)
\\
&\ge
e^{-\frac{c_3 c_4}{n\varepsilon_n} n } \left( e^{-n(\gamma-\varepsilon)}
- e^{-(\gamma + 1)n} \right).
\end{align*}
By assumption $n\varepsilon_n \rightarrow \infty$ which implies
$c_3 c_4/n\varepsilon_n \rightarrow 0$; this and the last inequality say\\
${\mathbb E}_{x_n} 
\left[
1_{\Gamma_n}
e^{-\frac{c_3}{n\varepsilon_n} \tau_{0,n}} 
\right]$
cannot decay at an exponential rate faster than $\gamma - \epsilon$,
but this contradicts
\eqref{e:basic} because $\varepsilon_n \rightarrow 0.$
Then, there cannot be $\epsilon > 0$ and a sequence
$\{n_k\}$ for which \eqref{e:contra} holds and this implies
the statement of Proposition \ref{t:thm1}.
\end{proof}

Define 
${\boldsymbol r}_3 \doteq \log(\rho_2) (1,1)$ and
$V(x) \doteq (-\log(\rho_1) + \langle {\boldsymbol r}_1, x\rangle )\wedge (-\log(\rho_2) + \langle {\boldsymbol r}_3, x \rangle)$,
for $x \in {\mathbb R}^2$ 
\begin{proposition}\label{p:ldfortandem}
\[
\lim_{n\rightarrow \infty}-\frac{1}{n} \log  P_{x_n}( \tau_n < \tau_0) = V(x)
\]
for $x \in {\mathbb R}_+^2$, $0 < x(1) + x(2) < 1$ and  $x_n = \lfloor nx \rfloor$.
\end{proposition}
The omitted proof is a one step version of the argument used in the proof
of Proposition \ref{t:thm1} and uses a mollification of $V$ as the subsolution.

\begin{proposition}\label{p:artikforXbar}
For any $\epsilon > 0$
there is  $N > 0 $ such that if $n > N$
\begin{equation}\label{e:artikforXbar}
P_x( \sigma_1 < \sigma_{1,2} < \tau < \infty)  \le e^{-n(\gamma-\epsilon)}
\end{equation}
where $x_n = \lfloor nx\rfloor$ and $x  \in {\mathbb R}_+^2$, $x(1) + x(2) < 1.$
\end{proposition}
\begin{proof}
Write
\begin{align*}
P_x( \sigma_1 < \sigma_{1,2} < \tau < \infty) 
=P_x( \sigma_1 < \sigma_{1,2} < \bar{\sigma}_{1,2} < \tau < \infty)
+
P_x( \sigma_1 < \sigma_{1,2} < \tau <  \bar{\sigma}_{1,2}).
\end{align*}
The definitions of $X$ and $\bar{X}$ imply 
$\tau_0 \ge  \bar{\sigma}_{1,2}$. Then, if a sample path $\omega$ satisfies
$\sigma_1(\omega) < \sigma_{1,2}(\omega) < \tau(\omega) < \bar{\sigma}_{1,2}$, it must 
also satisfy $\sigma_1(\omega) < \sigma_{1,2}(\omega) < \tau_n(\omega) < \tau_0(\omega)$.
This and Proposition \ref{t:thm1} imply that there is an $N$ such that
\[
P_{x_n}( \sigma_1 < \sigma_{1,2} < \tau <  
\bar{\sigma}_{1,2}) \le e^{-n(\gamma -\epsilon)},
\]
for $n > N$.
On the other hand, Proposition \ref{p:exactformula}
and the Markov property of $\bar{X}$ imply
\[
P_{x_n}( \sigma_1 < \sigma_{1,2} < \bar{\sigma}_{1,2} < \tau < \infty) \le
c_5 e^{-n(\gamma - \epsilon)},
\]
for some constant $c_5 > 0$.
These imply \eqref{e:artikforXbar}. 
\end{proof}

\begin{proof}[Proof of Proposition \ref{t:guzel}]
Decompose $P_x( \tau_n < \tau_0)$ and $P_x( \bar{\tau} < \infty)$ as follows:
\begin{align}
P_{x_n} ( \tau_n < \tau_0) &= \label{e:dectaun} P_{x_n}( \tau_n < \sigma_1 < \tau_0)+
 P_{x_n}(\sigma_1 < \tau_n \le  \sigma_{1,2} \wedge \tau_0)\\
&~~~~~~~~~~~~~~~~~~~+ P_{x_n}( \sigma_1 < \sigma_{1,2} < \tau_n < \tau_0)\notag\\
P_{x_n}( \tau < \infty) &= \label{e:dectn} 
P_{x_n}( \tau < \sigma_1 ) + P_{x_n}(\sigma_1 < \tau < \sigma_{1,2} )
+ P_{x_n}( \sigma_1 < \sigma_{1,2} < \tau < \infty). 
\end{align}
By definition $X$ and $\bar{X}$ are identical until they hit $\partial_1$;
therefore $\{\tau_n < \sigma_1\} = \{\tau < \sigma_1\}$ and 
\begin{equation}\label{e:equalityoffirstterms}
P_{x_n}(\tau_n < \sigma_1) = P_{x_n}( \tau < \sigma_1).
\end{equation}

The processes $X$ and $\bar{X}$ begin to differ after they hit $\partial_1$;
but Proposition \ref{p:equalityofsums} says that the sums of their components
remain equal before time $\sigma_{1,2}$; this implies $\tau = \tau_n$ on
$\tau_n \le \sigma_{1,2}$ and therefore
\[
 P_{x_n}(\sigma_1 < \tau \le \sigma_{1,2} )=
P_{x_n}(\sigma_1 < \tau_n \le  \sigma_{1,2} \wedge \tau_0)
\]
This \eqref{e:equalityoffirstterms} and the decompositions \eqref{e:dectaun} and
\eqref{e:dectn} imply
\begin{align*}
|P_{x_n} ( \tau_n < \tau_0) - P_{x_n}( \tau < \infty) |
=
| P_{x_n}( \sigma_1 < \sigma_{1,2} < \tau_n < \tau_0)-
P_{x_n}( \sigma_1 < \sigma_{1,2} < \tau < \infty)|
\end{align*}
By Propositions \ref{t:thm1} and \ref{p:artikforXbar} for $\epsilon > 0$
arbitrarily small the right side of
the last equality is bounded above by $e^{- n(\gamma-\epsilon)}$
when $n$ is large.
On the other hand, Proposition \ref{p:ldfortandem} says for $\epsilon_0 > 0$
arbitrarily small
$P_{x_n}( \tau_n < \tau_0) \ge e^{-n(\gamma_1 +\epsilon_0)}$ for $n$ large
where $\gamma_1 \doteq V(x) < \gamma$.
Choose $\epsilon$ and $\epsilon_0$ to satisfy
\[
\gamma -\gamma_1 > \epsilon+ \epsilon_0.
\]
These imply that 
for $c_6= (\epsilon + \epsilon_0)+ \gamma_1 -\gamma < 0$
\[
\frac{|P_{x_n} ( \tau_n < \tau_0) - P_{x_n}( \tau < \infty) |}
{|P_{x_n} ( \tau_n < \tau_0)|} < e^{c_6 n}
\]
when $n$ is large; this is what we have set out to prove.
\end{proof}

It is possible to generalize Proposition \ref{t:guzel} in many directions.
In particular, one expects it to hold for any tandem walk of finite dimension 
with the same exit boundary; the proof will almost be identical
but requires a generalization of Proposition \ref{p:ldfortandem}, which, we
believe, will involve the same ideas given in its proof. We leave this
task to a future work.

\section{Numerical Example}\label{s:examples}
Proposition \ref{t:guzel} says that for $x \in {\mathbb R}_+^2$ and $x_n =
\lfloor nx \rfloor$, 
the relative error
\[
\frac{|W^*(T_n(x_n)) - P_{x_n}( \tau_n < \tau_0)|}{P_{x_n}(\tau_n < \tau_0)}
\]
decays exponentially in $n$. Let us see numerically how well this approximation works.
Set $\mu_1 = 0.4$, $\mu_2 = 0.5$, $\lambda = 0.1$ and $n=60$. 
In two dimensions, one can
quickly compute $P_{x_n}(\tau_n < \tau_0)$ by numerically iterating \eqref{e:linear}
and using the boundary conditions $V_{\partial A_n} = 1$ and $V(0) = 0$; we will 
call the result of this computation ``exact.''  Because both $W^*(T_n(x_n))$ 
and $P_x(\tau_n <\tau_0)$ decay exponentially in $n$, it is visually simpler to compare
\begin{equation}\label{e:defVn}
V_n \doteq -\frac{1}{n} \log P_x( \tau_n < \tau_0),\text{ and }
W_n \doteq -\frac{1}{n} \log W^*(T_n(x_n)).
\end{equation}

\ninsepsc{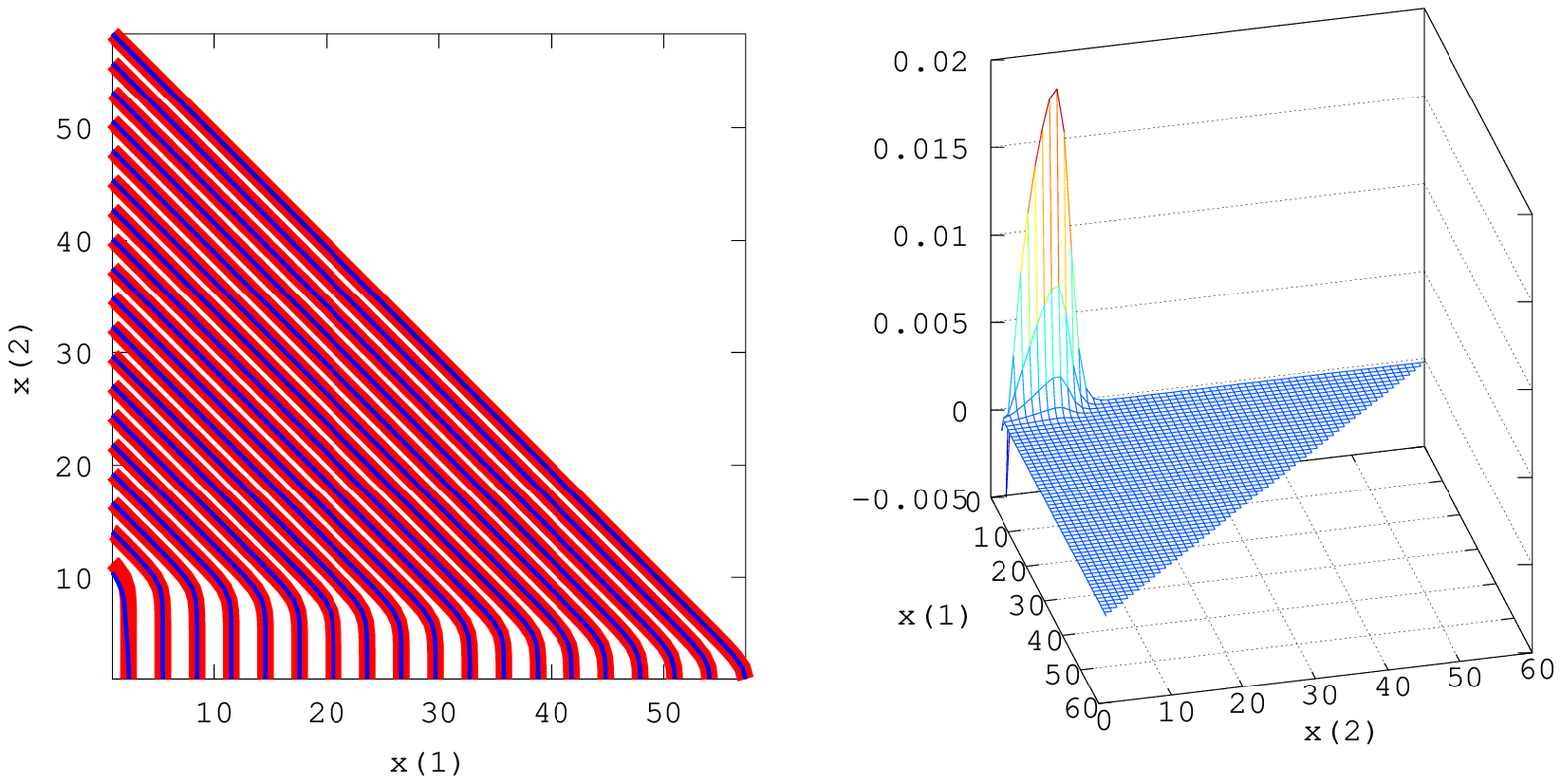}{On the left: level curves of $V_n$ (thin blue) and
$W_n$ (thick red); on the right: the graph of $(W_n-V_n)/W_n$ }{0.7}

The first graph in Figure \ref{f:relativeerror3} are the level curves
of $W_n$ of $V_n$; they all completely overlap except for the first one along the
$x(2)$ axis. The second graph shows the relative error $(W_n-V_n)/V_n$; we see that
it appears to be zero except for a narrow layer around $0$ where it is bounded by
$0.02$. 

For $x=(1,0)$,
the exact value for the
probability $P_x(\tau_{60} < \tau_0)$ is $1.1285 \cdot 10^{-35}$ and the approximate
value given by $W^*(T_n(x))$ equals $1.2037 \cdot 10^{-35}$.
Slightly away from the origin these quantities quickly converge to each other. For
example,
$P_x(\tau_{60} < \tau_0) = 4.8364 \cdot 10^{-35}$,
$W^*(T_n(x)) = 4.8148 \cdot 10^{-35}$ for $x=(2,0)$ and 
$P_x(\tau_{60} < \tau_0) = 7.8888 \cdot 10^{-31}$,
$W^*(T_n(x)) = 7.8885 \cdot 10^{-31}$ for $x = (9,0).$

\section{Literature Review}\label{s:review}
There is a vast literature related to the analysis presented in this
article.
Below we review a number of related works
and point out the connections between them and the present work.

There is a clear correspondence between the 
structures which appear in
the LD analysis and the subsolution approach to IS estimation of $p_n$
of \cite{thesis,DSW, dupuis2007subsolutions,sezer-asymptotically,yeniDW}
and those involved in the methods developed in this paper.
This connection is best expressed in the following equation (in
the context of two tandem walk just studied):
For $q=(q_1,q_2) \in {\mathbb R}^2$
set $\beta = e^{q_1}$ and $\alpha = e^{q_1-q_2}$; then
\[
H(q) = -\log ({\bm p}(\beta,\alpha)),
\]
where ${\bm p}$ is the characteristic polynomial 
defined in \eqref{e:hamiltonian}. A similar relation exists
between $H_2$ and ${\bm p}_2.$
In the LD analysis $H$ and $H_2$ appear as two of the
Hamiltonians of the limit deterministic continuous time
control problem; the gradient of the
limit value function lies on their zero level sets.
Parallel to our construction in subsection \ref{ss:asingleterm},
the articles using the subsolution approach 
construct subsolutions to a limit HJB equation using points on or inside the $0$ level
curve of the hamiltonians $H$ and $H_2$ or their intersection; for example, the gradient ${\bm r}_1$,
defined following display \eqref{e:subsolpa}.
lies exactly on this intersection and corresponds to the point $(\rho_1,\rho_1)$ 
lying on ${\mathcal H} \cap {\mathcal H}_2$; an example from prior work is given in
\cite[Figure 9]{DSW},
the point $r_1$ lying on the intersection of the $0$ level sets of the Hamiltonians $H$ and $H_2$ correspond
again
to the point $(\rho_1,\rho_1)$ lying on ${\mathcal H} \cap {\mathcal H}_2$
identified in subsection \ref{ss:asingleterm}). 
These works use subsolutions to estimate variances of
IS estimators (again based on the same subsolution) for  buffer overflow 
probabilities of the
form $P_x(\tau_n < \tau_0)$ and concentrate on the initial point $x=0$.
Concentrating on the initial points $x=0$ allows great
flexibility on the choice of the exit boundary 
$\partial A_n.$

In the present work we have studied the probability $P_x(\tau_n < \tau_0)$, 
which is a natural quantity to study if one is interested in the buffer 
overflow events of a queueing system.
Many other 
quantities naturally come up in the analysis of buffer overflows.
The work \cite{collingwood2011networks}, studies conditional
probabilities of the following form:
\begin{equation}\label{e:reversed}
P_0
( \sup_{t \in [0,T]} |X_{\lfloor nt\rfloor}/n - {\bm x}(t)| \le \epsilon
| \tau_{(n,0)} < \tau_0),
\end{equation}
where $X$ is the embedded random walk of a Jackson network with increments
$(0,1)$, 
$(-1,1)$,
$(1,-1)$,
$(0,-1)$ and $\tau_{(n,0)}$ is the first
time $X$ hits the point $(n,0)$ and ${\bm x}(\cdot)$ is a limit
deterministic process
to be computed; i.e., in studying
\eqref{e:reversed} one is interested in the behavior
of the queueing process conditioned on the rare event
$\{\tau_{n,0} < \tau_0 \}.$ 
The key idea in \cite{collingwood2011networks} and many
other works studying overflow events in queueing systems (see, e.g., the
list of references in \cite{collingwood2011networks})
is that if one chooses ${\bm x}$ in \eqref{e:reversed} to be the fluid 
limit of the time reversed process of $X$, the above conditional 
probabilities converge to $1$ (see \cite[Theorem 2]{collingwood2011networks}).
Then, in this line of analysis, the key steps are the computation
of the dynamics of the reversed process and its fluid limit. 
In computing these one needs the stationary
distribution of the $X$ process; an 
approximation of the stationary distribution is needed 
in cases when it is not
known exactly. The work \cite{collingwood2011networks} considers a modification
of the above two dimensional system, for which the stationary distribution
is not known and constructs approximations of its stationary distribution
of the form $\pi(x_k,y_k) \approx \frac{C}{1-\rho_1} L(y_n) \psi_0(x_k,y_k)$
as $x_k,y_k\rightarrow \infty$
where $C$ is a constant and $L$ is a harmonic function, 
not known explicitly but
its existence is guaranteed by results in \cite{foley2012constructing};
$\psi_0$
is constructed explicitly using harmonic functions of the form
$e^{\theta_1^* x + \theta_2 y}$ of the unconstrained version of
the random walk of interest.

The work \cite{mcdonald1999asymptotics} 
considers 
the buffer overflow of a chosen node 
in a given stable network.
The process considered in \cite{mcdonald1999asymptotics} is
$r+m$ dimensional: the first dimension represents the node whose
overflow event is to be studied, the dimensions $2,3,..,r$ represent
nodes that become unstable when the first node overflows, and the $m$
dimensions $r+1$,...,$r+m$, represent the ``super-stable'' nodes.
The analysis of \cite{mcdonald1999asymptotics} is based on the
$h$-transform of the embedded random walk of the queueing
system with the modification that its 
constraints are removed for the non super-stable dimensions,
i.e., the first $r$ dimensions (this process is denoted $W^{\infty}$),
the $h$ function is an harmonic function of the $W^\infty$ process
and is taken to be of the form $e^{\alpha x_1}a(x_2,...,x_{r+m})$;
\cite{mcdonald1999asymptotics}
gives conditions under which such
an $h$ function exists based on results from \cite{ney1987markov}.
For $n > 0$,
let $\tau_n$ be the
first time the first component of $W$ hits $n$, i.e, $\tau_n = \inf\{k: W(k) \in F_n\}$,
$F_n = \{x \in {\mathbb Z}_+^{r+m}: x_1 \ge n \}$;
let $\tau_0$ denote the first time $W$ hits the origin ${\bm 0}$. Finally, let $\tau_\triangle$ denote the first
time after time $0$, one of the nodes from $1$ to $r$ hits $0$, i.e., $\tau_\triangle = \inf\{k: k > 0, W \in {\small \triangle} \}$, 
where $\triangle = \{x: x_j = 0, \text{ for some 
}, j \in \{1,2,3,...,r\}\}$ is the constraining boundary of the state space for the components $1$ to $r$; remember
that these are the nodes that are assumed to become unstable when the first component overflows.
As an intermediate step in its analysis, \cite{mcdonald1999asymptotics} derives the following approximation
result: 
let $\pi_{\triangle}$ denote the
stationary measure conditioned on $\triangle$ 
and ${\mathbb E}_{\pi_\triangle}$ denote
expectation conditioned on $W(0)$ having initial distribution 
$\pi_{\triangle}$.
Let ${\bm \tau}_{\bm 0}$ be the first return time to ${\bm 0}$, i.e., ${\bm \tau}_{\bm 0} = \inf\{k > 0: W_k = {\bm 0}\}.$
\cite[Lemma 1.8]{mcdonald1999asymptotics} 
states,
under the assumptions made in the paper, 
\begin{equation}\label{e:approx}
\lim_{n\rightarrow \infty} \frac{|
 \pi({\bm 0})P_{{\bm 0}}(\tau_n < {\bm \tau}_{\bm 0}) - \pi(\Delta) P_{\pi_{\Delta}}(\tau_n < \tau_\triangle )|}{
 \pi({\bm 0})P_{{\bm 0}}(\tau_n < {\bm \tau}_{\bm 0})} = 0.
\end{equation}
\cite{mcdonald1999asymptotics} develops the following representation for 
 $\pi(\Delta) P_{\pi_{\Delta}}(\tau_n < \tau_\triangle )$:
\begin{equation}\label{e:formulaforbn}
 \pi(\Delta) P_{\pi_{\Delta}}(\tau_n < \tau_\triangle ) = e^{-\alpha n} {\mathbb E}_{\pi_\Delta}[h(W(1))\Psi(W(1))],
\end{equation}
$\Psi$ is defined as follows:
\begin{equation}\label{e:defpsi}
\Psi(x) = 
{\mathbb E}_x
[\hat{a}^{-1}(\hat{\mathscr W}^\infty(\tau_n)) 
e^{-\alpha ({\mathscr W}^\infty(\tau_n)-n)} 1_{\{\tau_n < \tau_{\blacktriangle}\}}],
\end{equation}
where, $\blacktriangle = \{x \in {\mathbb Z}^r_+ \times {\mathbb Z}^{m}, x_j \le 0, j \in \{1,2,3,...,r\}\}$,
${\mathscr W}^\infty$ is the $h$-transform of the process $W^\infty$.
For the computation of the expectation part of the formula \eqref{e:formulaforbn}, \cite{mcdonald1999asymptotics} suggests
simulation.
The seven conditions (see \cite[page 113, introduction]{mcdonald1999asymptotics}) 
that \cite{mcdonald1999asymptotics} is based on are conditions on the twisted
process, the stationary distribution $\varphi$ of its last $m$ components and on the stationary distribution $\pi$ of
the original process.
\cite[Section 3]{mcdonald1999asymptotics} treats the two dimensional 
constrained random walk
on ${\mathbb Z}_+^2$ with increments $(-1,0)$, $(1,0)$, $ (0,-1)$, 
$(0,1)$, $(1,1)$; for this process \cite{mcdonald1999asymptotics} 
constructs explicitly
an $h$ function of the form $h(x) =a_1^{x_1} a_2^{x_2}$, 
where $(a_1,a_2) \in {\mathbb R}^2$ is a point on a curve whose
definition is analogous to the definition of the characteristic surface 
${\mathcal H}.$

The work \cite{miyazawa2009tail} employs the
ideas of removing constraints on one of the boundaries and using points on curves associated with the resulting process to study
the tail asymptotics of the stationary distribution of a two dimensional nearest neigbor random walk ${\bm L}$ constrained
to remain in ${\mathbb Z}_+^2.$ To study the asymptotic decay rate of ${\bm \nu}(n,k)$ in $n$ for a fixed $k$, \cite{miyazawa2009tail}
considers the random walk ${\bm L}^{(1)}$, which has the same dynamics as  ${\bm L}$ except that it is not constrained on the vertical axis.
Associated with this process, \cite{miyazawa2009tail} defines two curves, whose definitions are parallel to the definition
of ${\mathcal H}$ and ${\mathcal H}_1$ (see the definition of ${\mathscr D}_1$ on \cite[page 554]{miyazawa2009tail}) and uses
points on and inside these curves to define solutions to an 
eigenvalue / eigenvector problem associated with the problem
(see \cite[Theorem 3.1]{miyazawa2009tail}); for the study of tail asymptotics along the vertical axis, \cite{miyazawa2009tail}
uses the same analysis but this time removing the constraint on the horizontal axis. For further works along this line
of research we refer the reader to \cite{kobayashi2013revisiting,dai2011reflecting,miyazawa2011light}.

The work \cite{ignatiouk2000large} develops an explicit formula for the large deviation local rate function $L(x,v)$ 
of a general Jackson network, starting from representations of these rates as limits derived in \cite{dupuis1995large,atar1999large}. 
For this, \cite{ignatiouk2000large} employs ``free processes;'' 
these are versions of the original process obtained 
by removing those constraints from the original process
that are not involved in a given direction $v$ at a given point $x \in {\mathbb R}_+^d.$ The proofs in \cite{ignatiouk2000large}
use fluid limits for the free process under a change of measure (i.e., a twisted/h-transformed version of the free process); 
the changes of measures used here correspond to using
$h$-functions of the form $e^{\langle \theta, x \rangle}$ where $\theta$ 
is a point on a characteristic surface (analogous to ${\mathcal H}$ in this work or $H$ in 
\cite{DSW}) associated with the process being transformed (see \cite[Section 6]{ignatiouk2000large}). As an application
of its results, \cite{ignatiouk2000large} computes the limit
$\lim_{n\rightarrow \infty} \frac{1}{n}\log {\mathbb E}_0[\tau_n]$
by noting from \cite{parekh1989quick} that this limit equals
\[
\lim_{n\rightarrow \infty} -\frac{1}{n}\log P_0( \tau_n < \tau_0),
\]
which is the LD decay rate
of the probability we have studied in this paper for general stable Jackson networks; \cite{ignatiouk2000large}
derives the explicit formula $\min_{1 \le i \le d} -\log(\rho_i)$ for the above LD rate using the explicit local rate functions
developed in the same work and the explicit formulas available for the stationary distribution of the underlying process.

The Martin boundary of an unstable process is a characterization of the 
directions through which the process may diverge to $\infty.$
The idea of using points on characteristic surfaces, and the idea of removing constraints from the process to simplify
analysis, appear also in works devoted to identifying Martin boundaries of constrained or stopped processes. An example is
\cite{ignatiouk2010martin}, which identifies the Martin boundary of two dimensional random walks
in ${\mathbb Z}_+^2$ and which are stopped as soon as they hit the boundary of ${\mathbb Z}_+^2$. This work
breaks up its analysis into three cases:
1)the directions $q \in {\mathbb R}_+^2$, where both components of $q$ are nonzero, 2) the directions $q$ such that
$q(1) = 0$, and 3) directions such that $q(2) = 0$. For each of these cases, \cite{ignatiouk2010martin} work with
what it calls {\em local processes}; the local process for the first case is a completely unconstrained random walk,
the local process for the second case is a process keeping the horizontal axis (i.e., the vertical boundary
is removed) and the third case is the reverse of the last. \cite{ignatiouk2010martin} uses LD analysis of the
local processes, harmonic functions of the form
\[
h_a(x) = \begin{cases} x_1 e^{\langle a, x\rangle} - {\mathbb E}_x[S_1(\tau) e^{\langle a,x\rangle}1_{\{\tau <\infty\}}], &\text{ if }  q(a) = (0,1),\\
x_2 e^{\langle a, x\rangle} - {\mathbb E}_x[S_2(\tau) e^{\langle a,x\rangle}1_{\{\tau <\infty\}}], &\text{ if }  q(a) = (1,0),\\
e^{\langle a, x\rangle} - {\mathbb E}_x[e^{\langle a,x\rangle}1_{\{\tau <\infty\}}], &\text{ otherwise.}
\end{cases}
\]
where $S$ is the underlying process, $\tau$ is the first hitting time to the boundary of ${\mathbb Z}_+^2$,
$a$ is a given point on a surface associated with $S$ (defined analgous to ${\mathcal H}$), 
$q(a)$ is the mean direction of $S$ under an exponential
change of measure defined by $a$ (see \cite[page 1108]{ignatiouk2010martin}.
In this connection let us also cite \cite{kurkova1998martin}, which uses geometry and complex analysis to identify the Martin boundary
of random walks on ${\mathbb Z}^2$, ${\mathbb Z} \times {\mathbb Z}_+$ and ${\mathbb Z}_+^2.$

Let $X$ be the constrained random walk in ${\mathbb Z}_+^2$ with increments $(1,0)$, $(-1,0)$, $(0,1)$, and $(0,-1)$
and let $\tau_n$ be as in \eqref{d:taun}.
A classical problem in computer science going back to 
\cite[section 2.2.2, exercise 13]{knuth1972art} is the analysis of the following expectation:
\begin{equation}\label{e:expectedexitpoint}
{\mathbb E}
\left[ \max(X_1(\tau_n),X_2(\tau_n))
\right],
\end{equation}
i.e., the expected size of the 
longest queue at the time of buffer overflow.  
This expectation is computed
in \cite{knuth1972art} for the case
$P(I_k=(1,0)) =P(I_k=(0,1)) = 1/2$, $P(I_k=(-1,0)) =P(I_k=(0,-1)) = 0$.
Various versions of this problem has since
been treated in 
\cite{yao1981analysis,flajolet1986evolution, 
maier1991colliding, comets2009large, louchard1994random, 
guillotin2006dynamic}. 
\cite{maier1991colliding} treats a generalization of this problem where the dynamics of the random walk
depend on its position; the approach of \cite{maier1991colliding} uses large deviations techniques from \cite{WF}.
\cite{yao1981analysis}
treats the approximation of \eqref{e:expectedexitpoint} for the case when the increments
have a symmetric distribution as follows: 
$P(I_k =(1,0)) = P(I_k(0,1)) = (1-p)/2$ and
$P(I_k =(-1,0)) = P(I_k(0,-1)) = p/2$; furthermore $p < 1/2$ is assumed, i.e., the process
is assumed unstable. Under these assumptions, \cite{yao1981analysis}
develops an approximation for the expectation in \eqref{e:expectedexitpoint}
as $n\rightarrow \infty.$ The main idea in \cite{yao1981analysis} is the
following: under the assumptions of the paper one can ignore both
of the constraining boundaries of the process, to prove this the author
uses LD bounds on iid Bernoulli sequences 
(see \cite[Lemma 3]{yao1981analysis}). Then an explicit
computation for the unconstrained process using elementary
techniques gives the desired approximation.

\section{Conclusion}\label{s:conclusion}
In this section
we point out several implications of our results, work in progress and possible extensions.
\subsection{The case $\mu_1 = \mu_2$}\label{ss:equal}
The formula \eqref{d:Wstar} 
for $P_y(\tau < \infty)$
(derived in Proposition \ref{p:exactformula})
requires $\mu_1 \neq \mu_2.$ The case $\mu_1 = \mu_2$ can be
handled by letting $\mu_2 \rightarrow \mu_1$ in \eqref{d:Wstar}; this
gives
\[
P_y(\tau < \infty) = \rho^{y(1)-y(2)} + \frac{\mu-\lambda}{\mu}\rho^{y(1)}
(y(1)-y(2)),
\]
where $\rho = \lambda / \mu$ and $\mu_1 = \mu_2 = \mu.$
Note that the case $\mu_1 = \mu_2$ leads to the linear term $y(1) -y(2)$.
\subsection[Constrained diffusions with drift]{Constrained diffusions with drift and
elliptic equations with Neumann boundary conditions}\label{ss:diffusionswithdrift}
Diffusion processes are weak limits of random walks. Thus, the results of
the previous sections can be used to compute/approximate Balayage and 
exit probabilities of constrained unstable diffusions. We give an example
demonstrating this possibility.

For $a,b > 0$
let $X$ be the the constrained diffusion on ${\mathbb R} \times {\mathbb R}_+$
with infinitesimal generator $L$ defined as
\[ f \rightarrow Lf,
Lf = \langle \nabla f,  ((2a+b), (a-b) ) \rangle + \frac{1}{6} \nabla^2 f \cdot  \left( \begin{matrix} 	
		2 & 1 \\
		1 & 2
\end{matrix}
\right),
\]
where $\nabla^2$
denotes the Hessian operator, mapping $f$ to its matrix of second order
partial derivatives.
On $\{x: x(2) = 0\}$ $X$ is pushed up to remain in ${\mathbb R} \times {\mathbb R}_+$
(the precise definition involves the Skorokhod map, see, e.g., \cite{kushner2001numerical}).
$a,b > 0$ implies that, starting from $B = \{x: x(1) > x(2)\}$, $X$ has positive probability
of never hitting $\partial B = \{x:x(1) = x(2)\}.$
Let $\tau$ be the first time $X$ hits $\{x:x(1) = x(2)\}$.
Proposition \ref{p:exactformula} for $d=2$ suggests
\begin{align}\label{e:exactforumlafordiffusions}
P_x(\tau < \infty) &= 
e^{-(a + 2b) 3(x(1)-x(2))} + \notag
\frac{a+2b}{a-b} e^{-(a + 2b) 3(x(1)-x(2))} e^{-(2a + b) 3 x(2)}\\
&~~-\frac{a+2b}{a-b}e^{-3(2a+b) x(1)},
x \in B.
\end{align}
One can check directly that the right side of the last display satisfies
\[
L V= 0,~~ \langle \nabla V, (0,1) \rangle = 0, x \in \partial_2.
\]
This and a verification argument similar to the proof of 
Proposition
\ref{p:balayagesimple} will imply \eqref{e:exactforumlafordiffusions}.
\subsection{General Jackson networks}
\paragraph{Multiple approximations}
We have seen with Proposition \ref{t:guzel} 
that $P_{y_n}(\tau < \infty)$
approximates $P_{x_n}(\tau_n < \tau_0)$ ,$x_n = \lfloor nx \rfloor$
very well (i.e., with exponentially decaying relative error)
for all $x \in A \doteq \{ x \in {\mathbb R}_+^2, 0 < x(1) + x(2) < 1\}$ when $n$ is large. 
When $X$ is the constrained random walk associated with a general two dimensional
Jackson network, this will not be true in general and to get a good approximation across all $A$ we will have to use the
transformation $T_n^2(x) = (x(1),(n-x(2)))$ as well as $T_n.$ $T_n^2$ moves the origin of the coordinate
system to the corner $(0,n)$ of $\partial A_n.$ 
Thus, for general  two dimensional $X$, 
we will have to construct two limit processes
$Y^1$ and $Y^2$; $Y^1$ will be as above and $Y^2$ will be the limit of
$Y^{2,n} \doteq T_n^2(X)$; the limit probability will be, as before $P_y( \tau^2 < \infty)$ where
$\tau^2$ is the first time $Y^2$ hits $\partial B.$ 
In $d$, dimensions we will
have $d$ possible limit processes, 
one for each corner of $\partial A_n$ providing precise approximations
for initial points which lie away from the boundaries missing
in the limit problem.
For a numerical
example see 
subsection 8.2 of the preprint \cite{sezer2015exit}.
One work in progress, based on the approach of
Section \ref{s:convergence2}, gives details of these ideas
in the context of Jackson networks consisting of parallel queues.
The same work also considers the approxmation of the expectation
\eqref{e:expectedexitpoint} using the techniques of the present work.

\paragraph{Approximation of $P_y(\tau < \infty)$ in general}
Second issue is the generalization of the
computation of the limit probability $P_y(\tau < \infty).$ 
As we have seen
in Proposition \ref{p:exactformula}, in the case of two tandem queues,
it is possible to compute this probability exactly as the superposition of two $Y$-harmonic functions:
$[(\rho_1,\rho_1),\cdot]$ and $h_{\rho_2}$. For general two dimensional
Jackson networks,  superposition of these two
functions will only give an approximation of $P_y(\tau < \infty)$; to construct
better approxmations one will proceed as indicated in Remark \ref{r:r2} and
use a linear combination of finite number of functions 
in the class of $Y$-harmonic functions constructed in 
subsections \ref{ss:asingleterm} and \ref{ss:twoterms} to approximate the constant function
$1$ on the boundary $\partial B$;
the error made in this approximation on $\partial B$ 
will provide
an upperbound for the error made in the approximation of 
$P_y(\tau < \infty)$ for any $ y\in B$. 
The numerical example in
\cite[subsection 8.2]{sezer2015exit} also demonstrates this point.

\paragraph{$\partial B$-determined $Y$-harmonic functions}
In the above paragraph we have noted that in general,
to construct improved approximations of $P_y(\tau < \infty)$,
 we will need to use further $Y$-harmonic functions 
of the form
\[
h_\beta = \beta^{y(1)-y(2)} \left(  C(\beta,\alpha_2) \alpha_1^{y(2)} - C(\beta,\alpha_1) \alpha_2^{y(2)} \right)
\]
where $(\beta,\alpha_1)$ and $(\beta,\alpha_2)$ are conjagate and
$\Delta(\beta) \neq 0$. We know by Proposition \ref{p:balayagesimple} that $h_\beta$
is $\partial B$-determined, if $|\alpha_1|,|\alpha_2| \le 1$ and ,$|\beta| < 1.$
Suppose we fix $\alpha \in \{z \in {\mathbb C}, |z|=1\}$  and compute $\beta$ and $\alpha^*$ so
that $(\beta,\alpha)$ and $(\beta,\alpha^*)$ are conjugate ($\beta$ and $\alpha^*$ are computed by
solving the characteristic equation ${\bm p}=1$). In view of Proposition \ref{p:balayagesimple},
and in view of the fact that $h_\beta$ will be used in the approximation of a $\partial B$-determined
$Y$-harmonic function,
a natural question is the following:
under what conditions on the parameters of the model do $|\alpha^*| \le 1$ and $|\beta| < 1$
hold? This problem is studied for the general two dimensional
Jackson network in Section 4 of \cite{sezer2015exit} (in particular, see Proposition 4.12 and Proposition 4.13).
These propositions require simplifying conditions on the system parameters (e.g., see \cite[condition (56), page 18]{sezer2015exit}).
Derivation of more precise conditions remains an open problem.

\paragraph{Harmonic systems}
In subsection \ref{ss:graphs} we have pointed out that the
classes of $Y$-harmonic functions constructed in 
subsections \ref{ss:asingleterm}
and \ref{ss:twoterms} have graph representations, as shown in
Figure \ref{f:graphsharmtex}; we refer to these graphs and the system of equations
they represent as ``harmonic systems.'' 
It is possible to generalize these graphs
to walks in $d$ dimensions and corresponding to each solution to the
system of equations represented by the graph one can define
a $Y$-harmonic function; this is done in the preprint
\cite[Section 5]{sezer2015exit} (see Definitions 5.1 and 5.2, Proposition
5.2, generalizing Proposition \ref{p:harmonicYtwoterms2d},
Proposition 5.3 generalizing Proposition \ref{p:balayagesimple}).
\paragraph{$d$-tandem queues}
Remarkably, it turns out to be possible to define a class
of harmonic systems and explicitly solve them
to generalize the formula \eqref{d:Wstar} for $P_y(\tau < \infty)$ to 
$d$ tandem queues. This is done in Section 6 of \cite{sezer2015exit}.
As an example, let us consider $d=3$. To compute $P(\tau < \infty)$,
one uses, in addition to the graphs given in Figure \ref{f:graphsharmtex},
the graph given in Figure \ref{f:G3tex}.

\begin{figure}[h]
\begin{center}
\scalebox{0.8}{
\centerline{\input{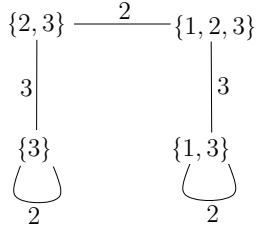}}}
\end{center}

\vspace{-0.65cm}
\caption{\hspace{0.25cm}An harmonic system for $d=3$\label{f:G3tex}}
\end{figure}

Proposition 6.3 of \cite{sezer2015exit} implies that, for 
\begin{equation}\label{e:nonequalmu}
\mu_i \neq \mu_j,
i,j \in \{1,2,3\},
\end{equation}
the following function solves the harmonic system given in
Figure \ref{f:G3tex}:
\begin{equation}\label{e:h3d}
h_{\rho_3}(y) = 
\rho_3^{y(1)-(y(2)+y(3))}\left(
1 - c_3 \rho_2^{y(3)} -c_3  c_{1} \rho_1^{y(2)} \rho_1^{y(3)} 
+c_3 c_{2} \rho_1^{y(2)} \rho_2^{y(3)}\right)
\end{equation}
where
\[
c_2 = \frac{\mu_2 - \lambda}{\mu_2 - \mu_1}, ~~ c_3 = \frac{\mu_3 - \lambda}{\mu_3 - \mu_2},
~~ c_1 = \frac{\mu_3 - \lambda}{\mu_3 - \mu_1}.
\]
The $Y$ process for the $3$-tandem queues is a random walk
on ${\mathbb Z}\times {\mathbb Z}_+^2$ with increments
$(-1,0,0)$, $(1,1,0)$, $(0,-1,1)$ and $(0,0,-1)$. $h$ of \eqref{e:h3d}
is a $Y$-harmonic function.
There are four terms in the sum \eqref{e:h3d} defining $h_{\rho_3}$, each of these
terms corresponds to a node of the graph in Figure \ref{f:G3tex}. None of
them is $Y$-harmonic individually. But the particular
linear combination in \eqref{e:h3d} is indeed $Y$-harmonic. Two further
$Y$-harmonic functions used in the calculation of $P_y(\tau < \infty)$ are
\[
h_{\rho_2} = \rho_2^{y(1)-(y(2)+y(3))} \left( \rho_2^{y(3)} -c_{2} \rho_1^{y(2)} \rho_2^{y(3)}\right), ~~~
h_{\rho_1} = \rho_1^{y(1)-(y(2)+y(3))} \rho_1^{y(2)} \rho_1^{y(3)};
\]
the harmonic systems for these functions are ``edge-completions'' of those
given in Figure \ref{f:graphsharmtex} (see Definition 5.4 of \cite{sezer2015exit}).
The exact formula for $P_y(\tau < \infty)$ for $y \in {\mathbb Z} \times {\mathbb Z}_+^2$, 
$y(1) \ge y(2) + y(3)$ is given in \cite[Proposition 6.5]{sezer2015exit} as
\[
P_y(\tau < \infty) =  h_{\rho_3} + c_3 h_{\rho_2} +  c_1 c_3 h_{\rho_1}.
\]
To treat the case when \eqref{e:nonequalmu} doesn't hold it suffices
to take limits in the last formula, which leads to polynomial
terms in $y$.


\subsection{Extension to other processes and domains}\label{ss:possiblegen}
In the foregoing sections, we have
approximated $P_x(\tau_n < \tau_0)$ 
in two stages: 1) use an affine
change of coordinates to move the origin to a point on the exit boundary
and take limits;
as a result, some of the constraints in the prelimit
process disappear and one obtains as a limit process an unstable constrained random walk and
as a limit problem the probability of return $P_y(\tau < \infty)$ of the unstable process;
2) find a class of basis functions on the exit boundary on which the 
Balayage operator of the limit process has a simple action; then 
try to approximate the function $1$ (i.e., the value of $P_y(\tau < \infty)$ on the exit boundary)
on the exit boundary with linear
combinations of the functions in the basis class.
The type of problem we have studied here is of the following form:
there is a process $X$ with a certain law of large number limit which 
takes $X$ away from a boundary $\partial A_n$ towards a stable point 
or a region; $\tau_0$ is the first time the process gets into this
stable region. We are interested in the probability $P(\tau_n < \tau_0)$.
We expect the first step to be 
applicable to a range of problems that fit into this scenario.
The second stage obviously depends on the 
particular dynamics of the original 
process. Ongoing research considers two tandem queues with
Markov modulated dynamics; optimal IS simulation for this
process was developed in \cite{sezer2009importance}. For Markov
modulated dynamics, one needs a more general class of $Y$-harmonic functions
than those constructed in Section \ref{s:twodim}
and the resulting equations are of higher degree and harder to analyze
but the main ideas of Section \ref{s:twodim} do generalize.
The present work focused on the exit boundary $\partial A_n$;
another natural exit boundary is $\{y: y(i) \le \lfloor a_i n \rfloor \}$ 
for $a_i > 0$, $i=1,2.$ We expect the analysis of this paper to generalize
to this exit boundary, with the following important modification: for
this boundary, there are three points on the exit boundary
from which one must conduct a limit analysis: the corners $n(0,a_2)$,
$n(0,a_1)$ and $n(a_1,a_2).$ For the last one the limit process will
be the completely unconstrained version of the random walk. Providing the
details of this and
further extensions to other processes and exit boundaries
remain problems for future research.

\bibliography{balayage}

\begin{thebibliography}{10}

\bibitem{alanyali1998large}
M.~Alanyali and B.~Hajek.
\newblock On large deviations in load sharing networks.
\newblock {\em Annals of Applied Probability}, pages 67--97, 1998.

\bibitem{aldous2013probability}
D.~Aldous.
\newblock {\em Probability approximations via the Poisson clumping heuristic},
  volume~77.
\newblock Springer Science \& Business Media, 2013.

\bibitem{anantharam}
J.~Anantharam, P.~Heidelberger, and P.~Tsoucas.
\newblock Analysis of rare events in continuous time {M}arkov chains via time
  reversal and fluid approximation.
\newblock {\em Tech Rep, IBM Research}, 1990.

\bibitem{asmussen2008applied}
S.~Asmussen.
\newblock {\em Applied probability and queues}, volume~51.
\newblock Springer Science \& Business Media, 2008.

\bibitem{asmussen2007stochastic}
S.~Asmussen and P.~Glynn.
\newblock {\em Stochastic simulation: Algorithms and analysis}, volume~57.
\newblock Springer Science \& Business Media, 2007.

\bibitem{atar1999large}
R.~Atar and P.~Dupuis.
\newblock Large deviations and queueing networks: methods for rate function
  identification.
\newblock {\em Stochastic processes and their applications}, 84(2):255--296,
  1999.

\bibitem{blanchet2013optimal}
J.~Blanchet.
\newblock Optimal sampling of overflow paths in jackson networks.
\newblock {\em Mathematics of Operations Research}, 38(4):698--719, 2013.

\bibitem{blanchet2006efficient}
J.~Blanchet, P.~Glynn, and K.~Leder.
\newblock {Efficient simulation of light-tailed sums: an old folk song sung to
  a faster new tune}.
\newblock {\em Monte Carlo and Quasi-Monte Carlo Methods 2008}, pages 227--258,
  2008.

\bibitem{blanchet2009lyapunov}
J.~Blanchet, P.~Glynn, and K.~Leder.
\newblock On lyapunov inequalities and subsolutions for efficient importance
  sampling.
\newblock 2009.
\newblock Preprint.

\bibitem{blanchet2009rare}
J.~Blanchet and M.~Mandjes.
\newblock Rare event simulation for queues.
\newblock {\em Rare Event Simulation Using Monte Carlo Methods}, pages 87--124,
  2009.

\bibitem{Rubetal04}
P.-T.~D. Boer, D.~P. Kroese, and R.~Y. Rubenstein.
\newblock A fast cross-entropy method for estimating buffer overflows in
  queueing networks.
\newblock {\em Management Science}, 50:883--895, 2004.

\bibitem{BoerNicola02}
P.-T.~D. Boer and V.~F. Nicola.
\newblock Adaptive state-dependent importance sampling simulation of
  {M}arkovian queueing networks.
\newblock {\em European Transactions on Telecommunications}, 13:303--315, 2001.

\bibitem{borovkov2001large}
A.~A. Borovkov and A.~A. Mogul'skii.
\newblock Large deviations for markov chains in the positive quadrant.
\newblock {\em Russian Mathematical Surveys}, 56(5):803--916, 2001.

\bibitem{MR1736592}
M.~Bou{\'e}, P.~Dupuis, and R.~S. Ellis.
\newblock Large deviations for small noise diffusions with discontinuous
  statistics.
\newblock {\em Probab. Theory Related Fields}, 116(1):125--149, 2000.

\bibitem{Changetal}
C.-S. Chang, P.~Heidelberger, S.~Juneja, and P.~Shahabuddin.
\newblock Effective bandwith and fast simulation of {ATM} intree networks.
\newblock {\em Performance Evaluation}, 20:45--66, 1994.

\bibitem{chen2013fundamentals}
H.~Chen and D.~Yao.
\newblock {\em Fundamentals of queueing networks: Performance, asymptotics, and
  optimization}, volume~46.
\newblock Springer Science \& Business Media, 2013.

\bibitem{collingwood2011networks}
J.~Collingwood, R.~D. Foley, and D.~R. McDonald.
\newblock Networks with cascading overloads.
\newblock In {\em Proceedings of the 6th International Conference on Queueing
  Theory and Network Applications}, pages 33--37. ACM, 2011.

\bibitem{comets2007distributed}
F.~Comets, F.~Delarue, and R.~Schott.
\newblock Distributed algorithms in an ergodic markovian environment.
\newblock {\em Random Structures \& Algorithms}, 30(1-2):131--167, 2007.

\bibitem{comets2009large}
F.~Comets, F.~Delarue, and R.~Schott.
\newblock Large deviations analysis for distributed algorithms in an ergodic
  markovian environment.
\newblock {\em Applied Mathematics and Optimization}, 60(3):341--396, 2009.

\bibitem{Iglehart74}
M.~A. Crane and D.~L. Iglehart.
\newblock Simulating stable stochastic systems, i: General multiserver queues.
\newblock {\em Journal of the Association for Computing Machinery},
  21(1):103--113, 1974.

\bibitem{dai2011reflecting}
J.~G. Dai, M.~Miyazawa, et~al.
\newblock Reflecting brownian motion in two dimensions: Exact asymptotics for
  the stationary distribution.
\newblock {\em Stochastic Systems}, 1(1):146--208, 2011.

\bibitem{de2006analysis}
P.-T. de~Boer.
\newblock Analysis of state-independent importance-sampling measures for the
  two-node tandem queue.
\newblock {\em ACM Transactions on Modeling and Computer Simulation (TOMACS)},
  16(3):225--250, 2006.

\bibitem{dean2009splitting}
T.~Dean and P.~Dupuis.
\newblock Splitting for rare event simulation: A large deviation approach to
  design and analysis.
\newblock {\em Stochastic processes and their applications}, 119(2):562--587,
  2009.

\bibitem{dieker2005asymptotically}
A.~T. Dieker and M.~Mandjes.
\newblock On asymptotically efficient simulation of large deviation
  probabilities.
\newblock {\em Advances in applied probability}, pages 539--552, 2005.

\bibitem{dupell2}
P.~Dupuis and R.~Ellis.
\newblock {\em A Weak Convergence Approach to the Theory of Large Deviations}.
\newblock John Wiley \& Sons, New York, 1997.

\bibitem{dupuis1995large}
P.~Dupuis and R.~S. Ellis.
\newblock The large deviation principle for a general class of queueing
  systems. i.
\newblock {\em Transactions of the American Mathematical Society},
  347(8):2689--2751, 1995.

\bibitem{MR1110990}
P.~Dupuis and H.~Ishii.
\newblock On {L}ipschitz continuity of the solution mapping to the {S}korokhod
  problem, with applications.
\newblock {\em Stochastics Stochastics Rep.}, 35(1):31--62, 1991.

\bibitem{KDW}
P.~Dupuis, K.~Leder, and H.~Wang.
\newblock Importance sampling for sums of random variables with regularly
  varying tails.
\newblock {\em ACM Trans. Model. Comput. Simul.}, 17(3):14, 2007.

\bibitem{DSW}
P.~Dupuis, A.~D. Sezer, and H.~Wang.
\newblock Dynamic importance sampling for queueing networks.
\newblock {\em Annals of Applied Probability}, 17(4):1306--1346, 2007.

\bibitem{duphui-is1}
P.~Dupuis and H.~Wang.
\newblock Importance sampling, large deviations and differential games.
\newblock {\em Stochastics and Stochastic Reports}, 76(6):481--508, 2004.

\bibitem{dupuis2007subsolutions}
P.~Dupuis and H.~Wang.
\newblock Subsolutions of an isaacs equation and efficient schemes for
  importance sampling.
\newblock {\em Mathematics of Operations Research}, 32(3):723, 2007.

\bibitem{yeniDW}
P.~Dupuis and H.~Wang.
\newblock Importance sampling for {J}ackson networks.
\newblock {\em Queueing Systems}, 62:113--157, 2009.

\bibitem{MR1609153}
R.~Durrett.
\newblock {\em Probability: theory and examples}.
\newblock Duxbury Press, Belmont, CA, second edition, 1996.

\bibitem{flajolet1986evolution}
P.~Flajolet.
\newblock {\em The evolution of two stacks in bounded space and random walks in
  a triangle}.
\newblock Springer, 1986.

\bibitem{foley2012constructing}
R.~D. Foley and D.~R. McDonald.
\newblock Constructing a harmonic function for an irreducible nonnegative
  matrix with convergence parameter r> 1.
\newblock {\em Bulletin of the London Mathematical Society}, page bdr115, 2012.

\bibitem{foley2005large}
R.~D. Foley, D.~R. McDonald, et~al.
\newblock Large deviations of a modified jackson network: Stability and rough
  asymptotics.
\newblock {\em The Annals of Applied Probability}, 15(1B):519--541, 2005.

\bibitem{frater1991optimally}
M.~R. Frater, T.~M. Lennon, and B.~D. Anderson.
\newblock Optimally efficient estimation of the statistics of rare events in
  queueing networks.
\newblock {\em IEEE Transactions on Automatic Control}, 36(12):1395--1405,
  1991.

\bibitem{WF}
M.~I. Freidlin and A.~D. Wentzell.
\newblock {\em Random Perturbation of Dynamical Systems, 2nd edition}.
\newblock Springer-Verlag Telos, 1998.

\bibitem{GlassKou}
P.~Glasserman and S.-G. Kou.
\newblock Analysis of an importance sampling estimator for tandem queues.
\newblock {\em ACM Transactions on Modeling and Computer Simulation}, 5:22--42,
  1995.

\bibitem{griffiths}
P.~Griffiths.
\newblock {\em Introduction to Algebraic Curves}.
\newblock American Mathematical Society, 1989.

\bibitem{guillotin2006dynamic}
N.~Guillotin-Plantard and R.~Schott.
\newblock {\em Dynamic random walks: Theory and applications}.
\newblock Elsevier, 2006.

\bibitem{ignatiouk2000large}
I.~Ignatiouk-Robert.
\newblock Large deviations of jackson networks.
\newblock {\em Annals of Applied Probability}, pages 962--1001, 2000.

\bibitem{ignatiouk2010martin}
I.~Ignatiouk-Robert and C.~Loree.
\newblock Martin boundary of a killed random walk on a quadrant.
\newblock {\em The Annals of Probability}, pages 1106--1142, 2010.

\bibitem{ignatyuk1994boundary}
I.~Ignatyuk, V.~A. Malyshev, and V.~Scherbakov.
\newblock Boundary effects in large deviation problems.
\newblock {\em Russian Mathematical Surveys}, 49(2):41--99, 1994.

\bibitem{JunejaNicola}
S.~Juneja and V.~Nicola.
\newblock Efficient simulation of buffer overflow probabilities in {J}ackson
  networks with feedback.
\newblock {\em ACM Transcations on Modeling and Computer Simulation},
  15:281--315, 2005.

\bibitem{juneja2006rare}
S.~Juneja and P.~Shahabuddin.
\newblock Rare-event simulation techniques: an introduction and recent
  advances.
\newblock {\em Handbooks in operations research and management science},
  13:291--350, 2006.

\bibitem{knuth1972art}
D.~E. Knuth.
\newblock {\em Art of Computer Programming Volume 1: Fundamental Algorithms}.
\newblock Addison-Wesley Publishing Company, 1972.

\bibitem{kobayashi2013revisiting}
M.~Kobayashi and M.~Miyazawa.
\newblock Revisiting the tail asymptotics of the double qbd process: refinement
  and complete solutions for the coordinate and diagonal directions.
\newblock In {\em Matrix-Analytic Methods in Stochastic Models}, pages
  145--185. Springer, 2013.

\bibitem{KroeseNicola}
D.~P. Kroese and V.~Nicola.
\newblock Efficient simulation of {J}ackson networks.
\newblock {\em ACM Transactions on Modeling and Computer Simulation},
  12:119--141, 2002.

\bibitem{kurkova1998martin}
I.~Kurkova and V.~Malyshev.
\newblock Martin boundary and elliptic curves.
\newblock {\em Markov Process. Related Fields}, 4(2):203--272, 1998.

\bibitem{kushner2001numerical}
H.~Kushner and P.~Dupuis.
\newblock {\em Numerical methods for stochastic control problems in continuous
  time}, volume~24.
\newblock Springer Science \& Business Media, 2001.

\bibitem{louchard1991probabilistic}
G.~Louchard and R.~Schott.
\newblock Probabilistic analysis of some distributed algorithms.
\newblock {\em Random Structures \& Algorithms}, 2(2):151--186, 1991.

\bibitem{louchard1994random}
G.~Louchard, R.~Schott, M.~Tolley, and P.~Zimmermann.
\newblock Random walks, heat equation and distributed algorithms.
\newblock {\em Journal of Computational and Applied Mathematics},
  53(2):243--274, 1994.

\bibitem{maier1991colliding}
R.~S. Maier.
\newblock Colliding stacks: A large deviations analysis.
\newblock {\em Random Structures \& Algorithms}, 2(4):379--420, 1991.

\bibitem{maier1993large}
R.~S. Maier.
\newblock Large fluctuations in stochastically perturbed nonlinear systems:
  Applications in computing.
\newblock {\em arXiv preprint chao-dyn/9305009}, 1993.

\bibitem{mcdonald1999asymptotics}
D.~McDonald.
\newblock Asymptotics of first passage times for random walk in an orthant.
\newblock {\em Annals of Applied Probability}, pages 110--145, 1999.

\bibitem{miretskiy2008state}
D.~Miretskiy, W.~Scheinhardt, and M.~R.~H. Mandjes.
\newblock State-dependent importance sampling for a jackson tandem network.
\newblock 2008.

\bibitem{miyazawa2009tail}
M.~Miyazawa.
\newblock Tail decay rates in double qbd processes and related reflected random
  walks.
\newblock {\em Mathematics of Operations Research}, 34(3):547--575, 2009.

\bibitem{miyazawa2011light}
M.~Miyazawa.
\newblock Light tail asymptotics in multidimensional reflecting processes for
  queueing networks.
\newblock {\em Top}, 19(2):233--299, 2011.

\bibitem{ney1987markov}
P.~Ney and E.~Nummelin.
\newblock Markov additive processes i. eigenvalue properties and limit
  theorems.
\newblock {\em The Annals of Probability}, pages 561--592, 1987.

\bibitem{nicola2007efficient}
V.~Nicola and T.~Zaburnenko.
\newblock Efficient importance sampling heuristics for the simulation of
  population overflow in jackson networks.
\newblock {\em ACM Transactions on Modeling and Computer Simulation (TOMACS)},
  17(2):10, 2007.

\bibitem{parekh1989quick}
S.~Parekh and J.~Walrand.
\newblock A quick simulation method for excessive backlogs in networks of
  queues.
\newblock {\em IEEE Transactions on Automatic Control}, 34(1):54--66, 1989.

\bibitem{JunejaRandhawa}
R.~Randhawa and S.~Juneja.
\newblock Combining importance sampling and temporal difference control
  variates to simulate markov chains.
\newblock {\em ACM Transactions on Modeling and Computer Simulation},
  14(1):1--30, 2004.

\bibitem{revuz1984markov}
D.~Revuz.
\newblock {\em Markov Chains}.
\newblock North-Holland, 1984.

\bibitem{ridder2009importance}
A.~Ridder.
\newblock Importance sampling algorithms for first passage time probabilities
  in the infinite server queue.
\newblock {\em European Journal of Operational Research}, 199(1):176--186,
  2009.

\bibitem{robert2003stochastic}
P.~Robert.
\newblock {\em Stochastic networks and queues, Stochastic Modelling and Applied
  Probability Series, vol. 52}.
\newblock Springer, New York, 2003.

\bibitem{rubino2009rare}
G.~Rubino and B.~Tuffin.
\newblock {\em Rare event simulation using Monte Carlo methods}.
\newblock John Wiley \& Sons, 2009.

\bibitem{thesis}
A.~D. Sezer.
\newblock {\em Dynamic Importance Sampling for Queueing Networks, Ph.D.
  thesis}.
\newblock Brown University Division of Applied Mathematics, 2005.

\bibitem{istrees}
A.~D. Sezer.
\newblock Asymptotically optimal importance sampling for {J}ackson networks
  with a tree topology.
\newblock 2007.
\newblock Preprint. Available at \url{http://arxiv.org/abs/0708.3260 }.

\bibitem{sezer2009importance}
A.~D. Sezer.
\newblock Importance sampling for a markov modulated queuing network.
\newblock {\em Stochastic Processes and their Applications}, 119(2):491--517,
  2009.

\bibitem{sezer-asymptotically}
A.~D. Sezer.
\newblock {Asymptotically optimal importance sampling for Jackson networks with
  a tree topology}.
\newblock {\em Queueing Systems}, 64(2):103--117, 2010.
\newblock Longer (2007) version available at
  \url{http://arxiv.org/abs/0708.3260 }.

\bibitem{sezer2015exit}
A.~D. Sezer.
\newblock Exit probabilities and balayage of constrained random walks.
\newblock {\em arXiv preprint arXiv:1506.08674}, 2015.

\bibitem{MR1335456}
A.~Shwartz and A.~Weiss.
\newblock {\em Large deviations for performance analysis}.
\newblock Stochastic Modeling Series. Chapman \& Hall, London, 1995.
\newblock Queues, communications, and computing, With an appendix by Robert J.
  Vanderbei.

\bibitem{yao1981analysis}
A.~C. Yao.
\newblock An analysis of a memory allocation scheme for implementing stacks.
\newblock {\em SIAM Journal on Computing}, 10(2):398--403, 1981.

\end{thebibliography}

\end{document}